\documentclass[a4paper,11pt]{amsart}

\newfont{\cyr}{wncyr10 scaled 1100}

\usepackage[left=2.7cm,right=2.7cm,top=3.5cm,bottom=3cm]{geometry}
\usepackage{amsthm,amssymb,amsmath,amsfonts,mathrsfs,amscd}
\usepackage[latin1]{inputenc}
\usepackage[all]{xy}
\usepackage{latexsym}
\usepackage{longtable}

\theoremstyle{plain}
\newtheorem{theorem}{Theorem}[section]
\newtheorem{corollary}[theorem]{Corollary}
\newtheorem{lemma}[theorem]{Lemma}
\newtheorem{proposition}[theorem]{Proposition}
\newtheorem{conjecture}[theorem]{Conjecture}

\theoremstyle{definition}
\newtheorem{definition}[theorem]{Definition}

\newtheorem{examplewr}[theorem]{Examples}
\newtheorem{ass}[theorem]{Assumption}

\theoremstyle{remark}
\newtheorem{obswr}[theorem]{Observation}
\newtheorem{remarkwr}[theorem]{Remark}

\newenvironment{remark}{\begin{remarkwr}\begin{upshape}}{\end{upshape}\end{remarkwr}}

\newcommand{\Emb}{\operatorname{Emb}}
\newcommand{\cF}{\mathcal F}

\newcommand{\V}{\mathcal V}
\newcommand{\E}{\mathcal E}
\newcommand{\g}{\gamma}
\newcommand{\G}{\Gamma}
\newcommand{\Gap}{\Gamma_0^D(M\ell)}
\newcommand{\Ga}{\Gamma_0^D(M)}
\newcommand{\hGa}{\hat{\Gamma}^D_0(M)}
\newcommand{\Q}{\mathbb{Q}}
\newcommand{\Z}{\mathbb{Z}}
\newcommand{\F}{\mathbb{F}}

\newcommand{\C}{\mathbb{C}}

\newcommand{\PP}{\mathbb{P}}
\newcommand{\Sel}{\operatorname{Sel}}
\newcommand{\Sha}{\mbox{\cyr{X}}}
\newcommand{\Gal}{\operatorname{Gal\,}}
\newcommand{\GL}{\operatorname{GL}}

\newcommand{\Div}{\operatorname{Div}}

\newcommand{\End}{\operatorname{End}}

\newcommand{\ord}{{\operatorname{ord}}}
\newfont{\gotip}{eufb10 at 12pt}

\newcommand{\cO}{{\mathcal O}}

\newcommand{\om}{{\omega}}
\newcommand{\ra}{\rightarrow}
\newcommand{\lra}{\longrightarrow}

\newcommand{\SL}{{\mathrm {SL}}}

\newcommand{\Pic}{{\mathrm{Pic}}}

\newcommand{\R}{{\mathbb R}}
\newcommand{\M}{{\mathrm{M}}}

\newcommand{\PGL}{{\mathrm{PGL}}}

\newcommand{\m}{\mathfrak{m}}
\newcommand{\p}{\mathfrak{p}}
\newcommand{\e}{\epsilon}
\newcommand{\W}{\mathcal W}

\def \mint {\times \hskip -1.1em \int}

\newcommand{\T}{\mathbb T}

\DeclareMathOperator{\Hom}{Hom}

\newcommand{\res}{\mathrm{res}}

\newcommand{\fr}{\mathfrak}
\newcommand{\cl }{\mathcal}

\newcommand{\longmono}{\mbox{$\lhook\joinrel\longrightarrow$}}

\newcommand{\longepi}{\mbox{$\relbar\joinrel\twoheadrightarrow$}}

\newcommand{\smallmat}[4]{\bigl(\begin{smallmatrix}#1&#2\\#3&#4\end{smallmatrix}\bigr)}

\include{thebibliography}

\begin{document}

\title{Special values of $L$-functions and the arithmetic of Darmon points}

\author{Matteo Longo, Victor Rotger and Stefano Vigni}

\thanks{The research of the second author is financially supported by DGICYT Grant MTM2009-13060-C02-01.}

\begin{abstract}
Building on our previous work on rigid analytic uniformizations, we introduce Darmon points on Jacobians of Shimura curves attached to quaternion algebras over $\Q$ and formulate conjectures about their rationality properties. Moreover, if $K$ is a real quadratic field, $E$ is an elliptic curve over $\Q$ without complex multiplication and $\chi$ is a ring class character such that $L_K(E,\chi,1)\not=0$ we prove a Gross--Zagier type formula relating Darmon points to a suitably defined algebraic part of $L_K(E,\chi,1)$; this generalizes results of Bertolini, Darmon and Dasgupta to the case of division quaternion algebras and arbitrary characters. Finally, as an application of this formula, assuming the rationality conjectures for Darmon points we obtain vanishing results for Selmer groups of $E$ over extensions of $K$ contained in narrow ring class fields when the analytic rank of $E$ is zero, as predicted by the Birch and Swinnerton-Dyer conjecture.
\end{abstract}

\address{Dipartimento di Matematica Pura e Applicata, Universit\`a di Padova, Via Trieste 63, 35121 Padova, Italy}
\email{mlongo@math.unipd.it}
\address{Departament de Matem\`atica Aplicada II, Universitat Polit\`ecnica de Catalunya, C. Jordi Girona 1-3, 08034 Barcelona, Spain}
\email{victor.rotger@upc.edu}
\address{Departament de Matem\`atica Aplicada II, Universitat Polit\`ecnica de Catalunya, C. Jordi Girona 1-3, 08034 Barcelona, Spain}
\curraddr{Department of Mathematics, King's College London, Strand, London WC2R 2LS, United Kingdom}
\email{stefano.vigni@kcl.ac.uk}



\maketitle

\section{Introduction} \label{intro}

The purpose of this article is threefold. Firstly, following \cite{Dar1}, \cite{Das} and \cite{Gr} and building on our previous work \cite{LRV} on rigid analytic uniformizations, we introduce a special supply of points on Jacobians of Shimura curves which we call \emph{Darmon points}, after the foundational work \cite{Dar1} of Henri Darmon in his investigation of counterparts in the real setting of the theory of complex multiplication. To be in line with the current language, our points could also be called ``Stark--Heegner points'' (as in \emph{loc. cit.}), but we feel that the new terminology we adopt here is more representative of the genesis of our constructions. Secondly, we prove an avatar of the Gross--Zagier formula relating Darmon points to the special values of twists by ring class characters of base changes to real quadratic fields $K$ of $L$-functions of elliptic curves $E$ over $\Q$, provided the analytic rank of $E$ over $K$ is $0$. Finally, under this analytic condition we use this formula to prove vanishing results for (twisted) Selmer groups of elliptic curves over narrow ring class fields of real quadratic fields. Let us describe first the motivation and background and then our results more in detail.

Let $A_{/\Q}$ be an elliptic curve of conductor $N_A$. Throughout this work we stay for simplicity in the realm of elliptic curves, but the reader should find no difficulties in extending our statements to the more general setting of a modular abelian variety $A_{/\Q}$ associated with a normalized newform $f_A\in S_2(\Gamma_0(N))$ with trivial nebentypus and Fourier coefficients living in a totally real number field of arbitrary degree; the reader may consult \cite{GuQu} and the references therein for the necessary background. 

Let $K$ be a real quadratic field of discriminant $\delta_K$ with $(N_A,\delta_K)=1$. Assume that there exists a prime $\ell$ which is inert in $K$ and divides $N_A$ exactly. If one further assumes the \emph{Heegner condition} that all primes dividing $N_A/\ell$ be split in $K$ then the sign of the functional equation of the $L$-function $L_K(A,s)$ of $A$ over $K$ is $-1$ and the Birch and Swinnerton-Dyer conjecture predicts that the rank of the Mordell--Weil group $A(H)$ is at least $[H:K]$ for all (narrow) ring class fields $H$ of $K$.

Under these conditions, in \cite{Dar1} Darmon introduced a family of \emph{local} points on $A$ over the unramified quadratic extension of $\Q_\ell$ and conjectured that they are in fact \emph{global}. More precisely, he predicted that his points are rational over narrow ring class fields of $K$ and satisfy properties which are analogous to those enjoyed by classical Heegner points over abelian extensions of imaginary quadratic fields (see \cite{BD2} for results in this direction); these points should account for the expectations of high rank described above. Darmon's points were later lifted from elliptic curves to certain quotients of classical modular Jacobians by Dasgupta in \cite{Das}; this was achieved by proving a rigid analytic uniformization result for modular Jacobians which can be phrased as an equality of $\mathcal L$-invariants and turns out to be a strong form of a theorem of Greenberg and Stevens (\cite{GS}). Both Darmon's and Dasgupta's constructions, relying heavily on the theory of modular symbols, do not lend themselves to straightforward extensions to more general settings in which the sign of the functional equation of $L_K(A,s)$ is still $-1$ (so that a similar family of points should exist) but the Heegner condition is not verified (cf. \cite[Conjecture 3.16]{Dar} or below for details). To circumvent this problem, in \cite{Gr} M. Greenberg reinterpreted Darmon's theory in terms of group cohomology; this allowed him to conjecturally define local points on $A$, generalizing Darmon's original constructions to much broader arithmetic contexts. It must be noted that Greenberg's definitions are conditional on the validity of an unproved statement (\cite[Conjecture 2]{Gr}); this conjecture (over $\Q$) has been proved by the authors of the present paper in \cite{LRV} and, independently and by different methods, by Dasgupta and Greenberg in \cite{DG}. 

The main result of \cite{LRV}, of which Greenberg's conjecture is a corollary, provides an explicit rigid analytic uniformization of the maximal toric quotient of the Jacobian of a Shimura curve attached to a division quaternion algebra over $\Q$ at a prime dividing exactly the level, and can be viewed as complementary to the classical theorem of \v{C}erednik and Drinfeld that gives rigid uniformizations at primes dividing the discriminant. Moreover, it extends to arbitrary quaternion algebras the results of Dasgupta for classical modular curves. 

In order to describe the content of this article we need to introduce some notation, which will be used throughout our work. As above, let $K$ be a real quadratic field of discriminant $\delta_K$, which we embed into the real numbers by using one of its two archimedean places $\infty_1,\infty_2$, and let $\ell $ be a prime number that remains inert in $K$. Let $\cO_K$ be the ring of integers of $K$ and for every integer $c\geq1$ let $\cO_c=\Z+c\cO_K$ be the order of $K$ of conductor $c$. Setting $\widehat{\cO}_c:=\cO_c\otimes_\Z\widehat\Z$ (with $\widehat\Z$ being the profinite completion of $\Z$), let
\[ \Pic^+(\cO_c)=\widehat{\cO}_c^\times K_{\infty,+}^\times\backslash\mathbb A_K^\times/K^\times \]
be the narrow (or strict) class group of $\cO_c$, where $\mathbb A_K$ is the ring of adeles of $K$ and $K_{\infty,+}^\times$ is the connected component of the identity in $K^\times_{\infty_1}\times K^\times_{\infty_2}$. By class field theory, $\Pic^+(\cl O_c)$ is canonically isomorphic to the Galois group $G_c:=\Gal(H_c/K)$ where $H_c$ is the narrow ring class field of $K$ of conductor $c$.

Let $D\geq1$ be the square-free product of an even number of primes and $M\geq1$ be a positive integer prime to $D$ such that $\ell\nmid DM$. Let $X_0^D(M)$ and $X_0^D(M\ell)$ denote the Shimura curves attached to the indefinite quaternion algebra $B$ of reduced discriminant $D$ and choices of Eichler orders $R'\subset R$ of levels $M$ and $M\ell$, respectively (cf. \cite[Ch. IV]{Dar}).

In the first part of this paper we introduce local Darmon points on the $\ell$-new quotient $J_0^D(M\ell)^{\text{$\ell$-new}}$ of the Jacobian of $X_0^D(M\ell)$; if $A_{/\Q}$ is an elliptic curve of conductor $DM\ell$ then we know by modularity and by the Jacquet--Langlands correspondence that $A$ is a quotient of $J_0^D(M\ell)^{\text{$\ell$-new}}$ and our points lift from $A$ those defined by Greenberg. Following \cite{Dar1}, \cite{Das} and \cite{Gr}, we formulate global rationality and reciprocity conjectures for them. All definitions and conjectures, together with a quick review of the main results of \cite{LRV}, can be found in Section \ref{sec-das} (see, in particular, \S \ref{SH-subsection}). We remark that the constructions and the techniques introduced in \cite{LRV} and in the present paper have been used in \cite{LV2} to prove that linear combinations of Darmon points on elliptic curves weighted by certain genus characters of $K$ are rational over the genus fields of $K$ predicted by Conjecture \ref{SH-conjecture}. This extends to an arbitrary quaternionic setting the theorem on the rationality of Stark--Heegner points obtained by Bertolini and Darmon in \cite{BD2}, and at the same time gives evidence for the rationality conjectures formulated here and in \cite{Gr}.

A crucial role in the definition of Darmon points is played by the group
\[ \Gamma_\ell:=\bigl(R\otimes\Z[1/\ell]\bigr)^\times_1 \]
of elements of reduced norm $1$ of $R\otimes\Z[1/\ell]$, which can be embedded in $\SL_2(\Q_\ell)$ and we call the \emph{Ihara group} at $\ell$ (see \S \ref{ihara-subsec}). The abelianization $\Gamma_\ell^{\rm ab}$ of $\Gamma_\ell$ is well known to be finite, and we devote Section \ref{ihara-sec} to the study of its support. In the absence of the counterparts for Shimura curves associated with division quaternion algebras of the results proved by Ribet in \cite{Ri} (this being due to the lack of a full analogue for general Shimura curves of the so-called Ihara's Lemma for modular curves), we invoke a theorem of Diamond and Taylor (\cite{DT}) on the Eisenstein-ness of certain maximal ideals of Hecke algebras to get a bound on the support of $\Gamma_\ell^{\rm ab}$ which is fine enough for our arithmetic purposes. The reader can find all details in \S \ref{sec-size} (see, in particular, Theorem \ref{Ihara-1}), which may be of independent interest.

Let us now describe the main results of this article. Let $E_{/\Q}$ be an elliptic curve without complex multiplication of conductor $N=N_E$ prime to $\delta_K$ and denote by $f_0(q)=\sum_{n=1}^\infty a_nq^n$ the normalized newform of weight $2$ on $\Gamma_0(N)$ associated with $E$ by the Shimura--Taniyama correspondence. Let $L_K(E,s)=L_K(f_0,s)$ be the complex $L$-function of $E$ over $K$ and assume that
\begin{itemize}
\item the sign of the functional equation of $L_K(E,s)$ is $+1$.
\end{itemize}
This implies that $L_K(E,s)$ vanishes to even order (and is expected to be frequently non-zero) at the critical point $s=1$. This is equivalent to saying that the set of primes
\[ \Sigma:=\bigl\{q|N:\;\text{$\ord_q(N)$ is odd and $q$ is inert in $K$}\bigr\} \]
has \emph{even} cardinality (and is possibly empty). We shall further assume that $\ord_q(N)=1$ for all $q\in\Sigma$. Let $D$ be the product of the primes in $\Sigma$ (with $D:=1$ if $\Sigma =\emptyset$), then set $M:=N/D$.

Write $\widehat{G}_c$ for the group of complex-valued characters of $G_c$, fix $\chi\in\widehat{G}_c$ and let $L_K(E,\chi,s)$ be the twist of $L_K(E,s)$ by $\chi$. For the remainder of this article choose $c$ prime to $\delta_K N$. By \cite[Theorem 3.15]{Dar}, it follows from our running assumptions that the sign of the functional equation for $L_K(E,\chi,s)$ is $+1$ as well.

Write $\Z[\chi]$ for the cyclotomic subring of $\C$ generated over $\Z$ by the values of $\chi$. In Section \ref{Popa-section} we introduce the algebraic part $\cl L_K(E,\chi,1)\in\Z[\chi]_S$ of the special value $L_K(E,\chi,1)$, where $S$ is a certain auxiliary finite set of prime numbers. Such an algebraic part is defined in terms of a twisted sum of homology cycles associated with conjugacy classes of oriented optimal embeddings of $\cO_c$ into a fixed Eichler order of $B$ of level $M$. Thanks to previous work of Popa (\cite{Po}), it can be shown that $L_K(E,\chi,1)\not=0$ if and only if $\cl L_K(E,\chi,1)\not=0$ (cf. Theorem \ref{popa-thm}).

From now on assume that $L_K(E,\chi,1)\not=0$. Suppose that $p$ is a prime number fulfilling the conditions listed in Assumption \ref{ass}, which exclude only finitely many primes. In particular, $p$ is a prime of good reduction for $E$ such that $\cl L_K(E,\chi,1)$ is not zero modulo $p$. Corresponding to any such $p$, in \S \ref{admissible-subsec} we introduce the notion of $p$-admissible primes (usually simply called ``admissible''), which are certain primes not dividing $Np$ and inert in $K$. For a sign $\epsilon\in\{\pm\}$ and a suitable $p$-admissible prime $\ell$ we introduce a map
\[ \partial_\ell:J_\epsilon^{(\ell)}(K_\ell)\otimes\Z[\chi]_S\longrightarrow\Z[\chi]/p\Z[\chi]_S \]
(denoted by $\partial'_\ell\otimes\text{id}$ in \S \ref{reciprocity-subsec}) and a twisted sum of Darmon points $P_\chi^\epsilon\in J_\epsilon^{(\ell)}(K_\ell)\otimes\Z[\chi]_S$. Here $J_\epsilon^{(\ell)}$ is an abelian variety over $\Q$ whose existence is a conjectural refinement of our work in \cite{LRV} and which is predicted to be isogenous to $J_0^D(M\ell )^{\text{$\ell$-new}}$ (see \S \ref{rig-subs} and \S \ref{SH-subsection} for details). If $D=1$ (i.e., $B\simeq \M_2(\Q)$) then our Darmon points need to be replaced by the points on modular Jacobians defined by Dasgupta in \cite[\S 3.3]{Das} (see also \cite[\S 1.2]{BDD}).

Letting $[\star]$ denote the class of the element $\star$ in a quotient group and writing $t_\ell$ for the exponent of $\Gamma_\ell^{\rm ab}$, our Gross--Zagier type formula for the special value of $L_K(E,\chi,s)$ can then be stated as follows.

\begin{theorem} \label{rec-intro-thm}
The equality $\partial_\ell(P_\chi^\epsilon)=t_\ell\cdot[\cl L_K(E,\chi,1)]$ holds in $\Z[\chi]_S/p\Z[\chi]_S$.
\end{theorem}

This result extends the main theorem of \cite{BDD}, where a similar formula was proved for $D=1$ and $\chi$ trivial. The extension of \cite[Theorem 3.9]{BDD} to the case of $D=1$ and arbitrary characters is relatively straightforward, the only ingredient that needs to be added being a version of Popa's classical formula in the twisted setting. However, note that the methods of \cite{BDD} are heavily based on modular symbol constructions, while our proof for arbitrary $D>1$ relies on the techniques introduced in \cite{LRV}. A proof of Theorem \ref{rec-intro-thm}, which can also be viewed as a ``reciprocity law'' in the sense of \cite{BD1}, is given in Theorem \ref{rec-law}. As in \cite{BDD}, a key ingredient is a level raising result (Theorem \ref{cong-thm}) at the admissible prime $\ell$; more precisely, since $\ell$ is inert in $K$, the construction of Darmon points is available ``at level $M\ell$'', and the proof of the above theorem boils down to suitably relating Darmon points on $J_\epsilon^{(\ell)}$ to the class modulo $p$ of the algebraic part $\cl L_K(E,\chi,1)$. We devote Sections \ref{raising-sec} and \ref{formula-sec} to a careful analysis of these issues.

What makes the formula of Theorem \ref{rec-intro-thm} interesting, and especially useful for the arithmetic applications we are going to describe, is the fact that $p$ does not divide the integer $t_\ell$. The possibility of requiring such a non-divisibility for a $p$-admissible prime $\ell$ is non-trivial and rests on the results on the support of $\Gamma_\ell^{\rm ab}$ that, as already mentioned, we obtain in Section \ref{ihara-sec}.

We conclude this introduction by stating the main arithmetic consequences of Theorem \ref{rec-intro-thm}. Let $K'$ be an extension of $K$ contained in $H_c$ for some $c\geq1$ as before and let $L_{K'}(E,s)$ be the $L$-function of $E$ over $K'$ (recall that a necessary and sufficient condition for an abelian extension $K'$ of $K$ to be contained in a ring class field is that it be Galois over $\Q$ with the non-trivial element of $\Gal(K/\Q)$ acting on $\Gal(K'/K)$ by inversion). For any prime number $p$ let $\Sel_p(E/K')$ be the $p$-Selmer group of $E$ over $K'$. While Theorem \ref{rec-intro-thm} is of a genuinely local nature (that is, to obtain it we do not need to use any conjectural global property of Darmon points), to prove the following vanishing result (Theorem \ref{cons2-thm}) we have to assume the validity of Conjecture \ref{SH-conjecture}, which predicts that the Darmon points are rational over suitable (narrow) ring class fields of $K$.

\begin{theorem} \label{thm3-intro}
Assume Conjecture \ref{SH-conjecture}. If $L_{K'}(E,1)\not=0$ then
\[ \Sel_p(E/K')=0 \]
for all but finitely many primes $p$. In particular, $E(K')$ is finite.
\end{theorem}

Theorem \ref{thm3-intro} is a consequence of a vanishing result for $p$-Selmer groups of $E$ twisted by anticyclotomic characters (Theorem \ref{main-thm}), and the set of primes for which it is valid contains those satisfying Assumption \ref{ass}. Observe that this result, which is predicted by the conjecture of Birch and Swinnerton-Dyer, is (a strengthening of) the counterpart in the real quadratic setting of the main result of \cite{LV}, which was obtained (unconditionally) in the more classical context of imaginary quadratic fields and Heegner points. When $D=1$ the above theorem represents an explicit instance of the ``potential arithmetic applications'' of Theorem \ref{rec-intro-thm} which are alluded to by Bertolini, Darmon and Dasgupta in the introduction to \cite{BDD}. We refer the reader to \S \ref{applications-subsec} for other arithmetic consequences of Theorem \ref{rec-intro-thm} (e.g., twisted versions of the Birch and Swinnerton-Dyer conjecture for $E$ over $K'$ in analytic rank $0$).

\vskip 2mm

\noindent\emph{Notation and conventions.} Throughout our work we fix an algebraic closure $\bar\Q$ of $\Q$ and view all number fields as subfields of $\bar\Q$. If $F$ is a number field we write $\cl O_F$ and $G_F$ for the ring of integers and the absolute Galois group $\Gal(\bar\Q/F)$ of $F$, respectively, and denote by $F_v$ the completion of $F$ at a place $v$.

For all prime numbers $\ell$ we fix an algebraic closure $\bar\Q_\ell$ of $\Q_\ell$ and an embedding $\bar\Q\hookrightarrow\bar\Q_\ell$. Moreover, $\C_\ell$ denotes the completion of $\bar\Q_\ell$.

If $\ell$ is a prime then $\F_\ell$ is the finite field with $\ell$ elements. We sometimes write $\F_p$ in place of $\Z/p\Z$ when we want to emphasize the field structure of $\Z/p\Z$.

If $G$ is a profinite group and $M$ is a continuous $G$-module we let $H^1(G,M)$ be the first group of continuous cohomology of $G$ with coefficients in $M$. In particular, if $G$ is the absolute Galois group of a (local or global) field $F$ then we denote $H^1(G,M)$ also by $H^1(F,M)$.

Let $F$ be a number field, $p$ a prime number and $A_{/F}$ an abelian variety.  We write $A[p^n]$ for the $p^n$-torsion subgroup of $A(\bar\Q)$. As customary, we let $\Sel_{p^n}(A/F)$ be the $p^n$-Selmer group of $A$ over $F$, i.e. the subgroup of $H^1(F,A[p^n])$ consisting of those classes which locally at every place of $F$ belong to the image of the local Kummer map. If $A$ has good reduction at a prime ideal $\mathfrak q\subset\cO_F$ such that $\mathfrak q\nmid p$ we let $H^1_{\rm sing}(F_\mathfrak q,A[p])$ and $H^1_{\rm fin}(F_\mathfrak q,A[p])$ denote the singular and finite parts of $H^1(F_\fr q,A[p])$ as defined in \cite[\S 3]{LV}.

Finally, for any ring $R$ and any pair of maps $f:M\rightarrow N$, $g:P\rightarrow Q$ of $R$-modules we write $f\otimes g:M\otimes_RP\rightarrow N\otimes_RQ$ for the $R$-linear map obtained by extending additively the rule $m\otimes p\mapsto f(m)\otimes g(p)$.

\vskip 2mm

\noindent\emph{Acknowledgements.} It is a pleasure to thank Kevin Buzzard, Henri Darmon, Benedict Gross, Yasutaka Ihara, Alexandru A. Popa and Alexei Skorobogatov for enlightening discussions and correspondence which helped improve some of the results of this article. We also would like to thank the anonymous referee for several helpful remarks and suggestions. Last but not least, heartfelt gratitude goes to Frank Sullivan for his invaluable help which allowed the first named author to spend March 2010 in Barcelona, at a delicate stage of this project. 

The three authors thank the Centre de Recerca Matem{\`a}tica (Bellaterra, Spain) for its warm hospitality in Winter 2010, when part of this research was carried out.

\section{On Ihara's group} \label{ihara-sec}

\subsection{Basic definitions} \label{ihara-subsec}

As in the introduction, let $D\geq1$ be a square-free product of an \emph{even} number of primes and let $M\geq1$ be an integer coprime with $D$. Let $B$ be the (unique, up to isomorphism) indefinite quaternion algebra over $\Q$ of discriminant $D$. Let $R=R(M)$ be a fixed Eichler order of level $M$ in $B$ and write $\Gamma_0^D(M)$ for the group of norm $1$ elements in $R$. If $\ell\nmid DM$ is a prime number then let $R'=R(M\ell)\subset R$ be an Eichler order of level $M\ell$ contained in $R$ and let $\Gamma_0^D(M\ell)$ be the group of norm $1$ elements in $R'$.

Fix an isomorphism of $\Q_\ell$-algebras
\[ \iota_\ell:B\otimes_\Q\Q_\ell\overset\simeq\longrightarrow\M_2(\Q_\ell) \]
such that $\iota_\ell(R\otimes\Z_\ell)$ is equal to $\M_2(\Z_\ell)$ and $\iota_\ell(R'\otimes\Z_\ell)$ is equal to the subgroup of $\M_2(\Z_\ell)$ consisting of upper triangular matrices modulo $\ell$. Letting the subscript ``$1$'' denote elements of norm $1$, we define the \emph{Ihara group} at $\ell$ to be the subgroup of $\SL_2(\Q_\ell)$ given by
\[ \Gamma_\ell:=\bigl(R\otimes\Z[1/\ell]\bigr)^\times_1\;\overset{\iota_\ell}{\longmono}\;\SL_2(\Q_\ell). \]
It acts on Drinfeld's $\ell$-adic half-plane $\mathcal H_\ell:=\C_\ell-\Q_\ell$ with dense orbits. The study of $\Gamma_\ell$ (or, rather, of its abelianization) when $\ell$ varies over the set of primes not dividing $MD$ will be the goal of the next two subsections.

\subsection{Finiteness of $\Gamma_\ell^{\rm ab}$} \label{finite-subsec}

We begin our discussion with a direct proof of the finiteness of the abelianization $\Gamma^{\rm ab}_\ell$ of $\Gamma_\ell$ for all $\ell\nmid MD$, which is a well-known fact (cf., e.g., \cite{Ih}). The reader is referred to \cite[Ch. VIII and IX]{Mar} (in particular, to \cite[Proposition 5.3, p. 324]{Mar}) for general results of this type.

Before proving the proposition we are interested in, let us introduce some notation. Let
\[ \pi_1,\pi_2: X^D_0(M\ell)\longrightarrow X^D_0(M),\qquad\Gamma_0^D(M\ell)z\overset{\pi_1}{\longmapsto}\Gamma_0^D(M)z,\qquad\Gamma_0^D(M\ell)z\overset{\pi_2}{\longmapsto}\Gamma_0^D(M)\om_\ell z \]
be the two natural degeneracy maps. Here $\om_\ell$ is an element in $R(M\ell)$ of reduced norm $\ell$ that normalizes $\Gamma_0^D(M\ell)$. As a piece of notation, for any element $\g$ in (respectively, subgroup $G$ of) $\Gamma_0^D(M\ell)$ we shall write $\hat\g:=\om_\ell\g\om^{-1}_\ell$ (respectively, $\hat{G}:=\om_\ell G\om_\ell^{-1}$). Moreover, let
\[ \pi^\ast:=\pi_1^\ast\oplus\pi_2^\ast:H_1\bigl(X_0^D(M),\Z\bigr)^2\longrightarrow H_1\bigl(X_0^D(M\ell),\Z\bigr) \]
and
\[ \pi_\ast:=(\pi_{1,\ast},\pi_{2,\ast}):H_1\bigl(X_0^D(M\ell),\Z\bigr)\longrightarrow H_1\bigl(X_0^D(M),\Z\bigr)^2 \]
be the maps induced in homology by pull-back and push-forward, respectively. In terms of group homology, they correspond to the maps
\[ \pi^\ast:=\pi_1^\ast\oplus\pi_2^\ast:H_1\bigl(\Gamma_0^D(M),\Z\bigr)^2\longrightarrow H_1\bigl(\Gamma_0^D(M\ell),\Z\bigr) \]
and
\[ \pi_\ast:=(\pi_{1,\ast},\pi_{2,\ast}):H_1\bigl(\Gamma_0^D(M\ell),\Z\bigr)\longrightarrow H_1\bigl(\Gamma_0^D(M),\Z\bigr)^2 \]
induced by corestriction and restriction, respectively.

\begin{proposition} \label{finite-ab-prop}
The group $\Gamma^{\rm ab}_\ell$ is finite for all primes $\ell\nmid MD$.
\end{proposition}

\begin{proof}
As shown in \cite[equation (30)]{LRV}, there is a long exact sequence in homology
\begin{equation} \label{ex-seq-iso}
\begin{split}
H_1\bigl(\Gamma_0^D(M\ell),\Z\bigr)\overset{\pi_*}{\longrightarrow}H_1\bigl(\Gamma_0^D(M),\Z\bigr)^2&\longrightarrow H_1(\G_\ell,\Z)\\
&\longrightarrow H_0\bigl(\Gamma_0^D(M\ell),\Z\bigr)\longrightarrow H_0\bigl(\Gamma_0^D(M),\Z\bigr)^2.
\end{split}
\end{equation}
Since the actions on $\Z$ are trivial, the last homomorphism can be naturally identified with the diagonal embedding of $\Z$ into $\Z^2$, which is obviously injective. Thus the exactness of \eqref{ex-seq-iso} implies that $\mathrm{coker}(\pi_*)\simeq H_1(\G_\ell,\Z)$, which in turn is isomorphic to $\Gamma^{\rm ab}_\ell$. But in the proof of \cite[Lemma 6.2]{LRV} it is shown that the endomorphism $\pi_*\circ\pi^*$ is injective with finite cokernel. Since $\mathrm{coker}(\pi_*)$ is a quotient of $\mathrm{coker}(\pi_*\circ\pi^*)$, it follows that $\Gamma^{\rm ab}_\ell$ is finite. \end{proof}

\subsection{Results on the support of $\Gamma^{\rm ab}_\ell$} \label{sec-size}

In this subsection we study the support (i.e., the set of primes dividing the order) of $\Gamma^{\rm ab}_\ell$, which is finite by Proposition \ref{finite-ab-prop}, as $\ell$ varies in the set of primes not dividing $MD$. Thanks to Ihara's Lemma, in the case of modular curves (i.e., when $D=1$) the size of $\Gamma^{\rm ab}_\ell$ is controlled in \cite[Theorem 4.3]{Ri}, and an explicit result on the support of $\Gamma_\ell^{\rm ab}$ has been given by Dasgupta in \cite{Das}. Namely, in \cite[Proposition 3.7]{Das} it is shown that the primes in this set are divisors of $6\phi(M)(\ell^2-1)$ where $\phi$ is the classical Euler function.

Assume $D>1$. The extra difficulties in the non-split quaternionic setting arise from the fact that the counterpart of \cite{Ri} is not yet available. Results of this type would follow, for instance, if $\Gamma_\ell$ had the so-called ``congruence subgroup property''. In this case, it might be possible to show that the support of $\Gamma_\ell^{\rm ab}$ is contained in the set of primes dividing $\phi(M)$, thus showing that it is in fact independent of $\ell$. See \cite{CT} for an account of this problem.

We will obtain results on the support of $\Gamma^{\rm ab}_\ell$ by means of a theorem of Diamond and Taylor (\cite[Theorem 2]{DT}) which represents a weak analogue of Ihara's Lemma for Shimura curves.

To begin our study, observe that the two coverings $\pi_1$ and $\pi_2$ of \S \ref{finite-subsec} give rise by Picard functoriality to a homomorphism of abelian varieties
\[ \xi:J_0^D(M)\oplus J_0^D(M)\longrightarrow J_0^D(M\ell) \]
between Jacobians. The kernel of $\xi$ is isomorphic to $\Hom(\Gamma^{\rm ab}_\ell,{\bf U})$ where ${\bf U}$ is the group of complex numbers of norm $1$. Thus we see that if a prime number $p$ is in the support of $\Gamma^{\rm ab}_\ell$ then the map 
\[ \xi_p:J_0^D(M)[p]\oplus J_0^D(M)[p]\longrightarrow J_0^D(M\ell)[p] \]
induced by $\xi$ on the $p$-torsion subgroup is not injective. We study the kernel of $\xi_p$ by means of \cite[Theorem 2]{DT}.

To start with, let us fix some notation. For any prime $q\nmid D$ choose an isomorphism $\varphi_q:B\otimes_\Q\Q_q\simeq\M_2(\Q_q)$ of $\Q_q$-algebras in such a way that for all $q|M$ one has
\[ \varphi_q(R\otimes\Z_q)=\left\{\begin{pmatrix}a&b\\c&d\end{pmatrix}\in\M_2(\Z_q)\;\Big|\;c\equiv0\pmod{q^{n(q)}}\right\} \]
where $q^{n(q)}$ is the exact power of $q$ dividing $M$. We also require that $\varphi_\ell$ satisfies the additional condition
\[ \varphi_\ell(R^\prime\otimes\Z_\ell)=\left\{\begin{pmatrix}a&b\\c&d\end{pmatrix}\in\M_2(\Z_\ell)\;\Big|\;c\equiv0\pmod{\ell}\right\}. \]
For every $q$ as above and every integer $m\geq0$ write $\Gamma^{\rm loc}_0(q^m)$ for the subgroup of $\GL_2(\Z_q)$ consisting of matrices $\smallmat abcd$ with $c\equiv0\pmod{q^m}$. We further denote by $\Gamma_1^{\rm loc}(q^m)$ the subgroup of $\Gamma_0^{\rm loc}(q^m)$ consisting of matrices $\smallmat abcd$ with $d\equiv1\pmod{q^m}$ and $c\equiv0\pmod{q^m}$. For primes $q\nmid D$ let
\[ i_q:B\;\longmono\;\GL_2(\Q_q) \]
denote the composition of the canonical inclusion $B\hookrightarrow B\otimes\Q_q$ with isomorphism $\varphi_q$. Let $\Gamma_1^D(M)$ be the subgroup of $\Gamma_0^D(M)$ consisting of those elements $\gamma$ such that $i_q(\gamma)\in\Gamma_1^{\rm loc}(q^{n(q)})$ for all $q|M$. Moreover, let $Q\geq1$ be the smallest integer such that $MQ\geq4$ and $\ell\nmid Q$ (so $Q=1$ if $M\geq4$) and define $\Gamma_1^D(MQ)$ as the subgroup of $\Gamma_1^D(M)$ consisting of those elements $\gamma$ such that $i_q(\gamma)\in \Gamma_1^{\rm loc}(q)$. Finally, consider the subgroup $\Gamma_{1,0}^D(MQ,\ell)$ of $\Gamma_1^D(MQ)$ whose elements are the $\gamma$ such that $i_\ell(\gamma)\in \Gamma_0^{\rm loc}(\ell)$. Write $X_1^D(M)$, $X_1^D(MQ)$ and $X_{1,0}^D(MQ,\ell)$ for the compact Shimura curves associated with $\Gamma_1^D(M)$, $\Gamma_{1}^D(MQ)$ and $\Gamma_{1,0}^D(MQ,\ell)$, respectively, and let $J_1^D(M)$, $J_1^D(MQ)$ and $J_{1,0}^D(MQ,\ell)$ denote their Jacobian varieties. For $i=1,2$ the inclusion  $\Gamma_{1,0}^D(MQ,\ell)\subset\Gamma_1^D(MQ)$ induces coverings
\[ \vartheta_i:X^D_{1,0}(MQ,\ell)\longrightarrow X_{1}^D(MQ) \]
defined, as above, by $\vartheta_1([z])=[z]$ and $\vartheta_2([z])=[\omega_\ell(z)]$. By Picard functoriality, we obtain a homomorphism
\[ \vartheta:J_1^D(MQ)\oplus J_1^D(MQ)\longrightarrow J_{1,0}^D(MQ,\ell) \]
between Jacobians. Further, the inclusions
\[ \Gamma_1^D(MQ)\subset\Gamma_1^D(M)\subset\Gamma_0^D(M) \]
induce converings of the relevant Riemann surfaces and thus, again by Picard functoriality, homomorphisms $\sigma:J_0^D(M)\rightarrow J_1^D(M)$ and $\eta:J_1^D(M)\rightarrow J_1^D(MQ)$. Finally, the inclusion $\Gamma_0^D(M\ell)\subset\Gamma_{1,0}^D(MQ,\ell)$ gives a homomorphism $\rho:J_0^D(M\ell)\rightarrow J_{1,0}^D(MQ,\ell)$. These maps fit in the commutative diagram
\begin{equation} \label{diagram-X}
\xymatrix@R=30pt@C=30pt{J_0^D(M)\oplus J_0^D(M)\ar[r]^-{\sigma\oplus\sigma}\ar[d]^-\xi & J_1^D(M)\oplus J_1^D(M)\ar[r]^-{\eta\oplus\eta} & J_1^D(MQ)\oplus J_1^D(MQ)\ar[d]^-\vartheta\\
J_0^D(M\ell)\ar[rr]^\rho && J_{1,0}^D(MQ,\ell).}
\end{equation}
Since $\sigma$ and $\eta$ arise by Picard functoriality from coverings of Riemann surfaces, their kernels are finite. Thus the kernels of $\sigma\oplus\sigma$ and $\eta\oplus\eta$ are finite too, and we denote by $C_1$ and $C_2$ their orders. Note that $C_1$ and $C_2$ do not depend on $\ell$ (the kernel of $\sigma$ is, by definition, the {\it Shimura subgroup} of $J_0^D(M)$ and its size is known to divide $\phi(M)$: see \cite{L}).

Observe that the kernel of $\vartheta$ is finite as well. To show this, note that the maps $\vartheta_1$ and $\vartheta_2$ induce, this time by Albanese functoriality, a homomorphism
\[ \vartheta':J_{1,0}^D(MQ,\ell)\longrightarrow J_1^D(MQ)\oplus J_1^D(MQ) \]
on Jacobians, and the composition $\vartheta'\circ\vartheta$ is represented by the matrix $\smallmat {\ell+1}{T_\ell}{T_\ell}{\ell+1}$. Since the eigenvalues of $T_\ell$ are bounded by $2\sqrt\ell$, we see that $\vartheta'\circ\vartheta$ is injective on tangent spaces, and thus its kernel is finite. So the kernel of $\vartheta$ is finite; we denote its cardinality by $C(\ell)$. In the following we study the size of $C(\ell)$. We first note that if a prime $p$ divides $C(\ell)$ then the map
\[ \vartheta_p:J_1^D(MQ)[p]\oplus J_1^D(MQ)[p]\longrightarrow J_{1,0}^D(MQ,\ell)[p]\]
induced by $\vartheta$ on the $p$-torsion subgroup is not injective.

For any discrete subgroup $G$ of $\SL_2(\R)$ denote by $S_2(G,\C)$ the $\C$-vector space of cusp forms of weight $2$ and level $G$. Let $\mathcal F=\{f_1,\dots,f_h\}$, where $h$ is the dimension of $J_1^D(MQ)$, be a basis of $S_2\bigl(\Gamma_1^D(MQ),\C\bigr)$ consisting of eigenforms for the action of the Hecke algebra and (at the cost of renumbering) assume that $\{f_1,\dots,f_m\}$ is a set of representatives for the set of orbits of $\mathcal F$ under the action of $G_\Q$. Denote by $A_1=A_{f_1},\dots,A_m=A_{f_m}$ the abelian varieties associated with these forms via the Eichler--Shimura construction, fix an isogeny
\[ J_1^D(MQ)\overset\sim\longrightarrow\prod_{i=1}^mA_i \]
and let $C_3$ be the order of its kernel, which of course does not depend on $\ell$. By the Jacquet--Langlands correspondence, each of the abelian varieties $A_i$ is isogenous over $\Q$ to the abelian variety $A_{f_{0,i}}$ associated with a classical modular form $f_{0,i}\in S_2\bigl(\Gamma_1(MDQ),\C\bigr)$ for the congruence subgroup $\Gamma_1(MDQ)\subset\SL_2(\Z)$.

For every $i=1,\dots,m$ fix an isogeny $\psi_i:A_i\rightarrow A_{f_{0,i}}$ and denote by $d_i$ the size of its kernel. Set $C_4:=\prod_{i=1}^m d_i$ and notice that $C_4$ is independent of $\ell$. Finally, recall that the mod $p$ representation of $G_\Q$ associated with a modular form $f\in S_2\bigl(\Gamma_1^D(MQ),\C\bigr)$ is irreducible for all but finitely many prime numbers $p$. For every $i$ let $e_i$ be the product of the primes $p$ such that the $G_\Q$-representation $A_{f_{0,i}}[p]$ is reducible, then set $C_5:=\prod_{i=1}^me_i$.

Now let us recall \cite[Theorem 2]{DT}, which is a (weak) substitute for Ihara's Lemma in the context of
Shimura curves attached to non-split quaternion algebras. Let $p$ be a prime number not dividing $6MDQ\ell$.
Following \cite{DT}, denote by $\T$ the image in $\End\bigl(J_1^D(MQ)\bigr)$ of the polynomial ring generated over $\Z$ by the Hecke operators $T_q$ and the spherical (i.e., diamond) operators $S_q$ for primes $q\nmid MDQ$. A maximal ideal $\mathfrak m$ of $\T$ containing $p$ is said to be \emph{Eisenstein} if for some integer $d\geq1$ and all but finitely many primes $q$ with $q\equiv1\pmod{d}$ we have $T_q-2\in\mathfrak m$ and $S_q-1\in\mathfrak m$. By \cite[Theorem 2]{DT}, the maximal ideals of $\T$ in the support of $\ker(\vartheta_p)$ are Eisenstein.

If $\mathfrak m$ is a maximal ideal of $\mathbb T$ belonging to the support of $S_2\bigl(\Gamma_1(MDQ),\C\bigr)$ with residual characteristic $p$ then $\mathfrak m$ is the kernel of the reduction modulo $p$ of the homomorphism $\T\rightarrow\mathcal O_E$ associated with one of the eigenforms $f_{0,i}\in S_2\bigl(\Gamma_1(MDQ),\C\bigr)$, where $E$ is a suitable number field. For simplicity, denote by $f$ the eigenform associated with $\mathfrak m$. By \cite[Proposition 2]{DT}, the ideal $\mathfrak m$ is Eisenstein if and only if the mod $p$ Galois representation $\rho_\mathfrak m$ attached to $\mathfrak m$ is reducible. With notation as above, this can be rephrased by saying that $\mathfrak m$ is Eisenstein if and only if the $G_\Q$-representation $A_f[p]$ is reducible.

The main result of this subsection is the following

\begin{theorem} \label{Ihara-1}
There exists an integer $C\geq1$ such that for all but finitely many primes $\ell\nmid MD$ the support of $\Gamma^{\rm ab}_\ell$ is contained in the set of primes dividing $C\ell$.
\end{theorem}

\begin{proof} With notation as before, we show that the integer
\[ C:=6C_1C_2C_3C_4C_5MDQ, \]
which depends only on $M$, $D$ and $Q$, does the job. More precisely, we  show that if $\ell\nmid MDQ$ and the prime $p$ belongs to the support of $\Gamma^{\rm ab}_\ell$ then $p$ divides $C\ell$. Thus fix a prime $\ell\nmid MDQ$. As remarked earlier, if the prime $p$ lies in the support of $\Gamma_\ell^{\rm ab}$ then $\ker(\xi_p)$ is not zero. 

The first step of the proof consists in showing that if $p\nmid C_3C_4MDQ\ell$ but $p$ divides the order of $\ker(\vartheta_p)$ then $p|C_5$. To this aim, fix a maximal ideal $\mathfrak m$ of $\T$ in the support of $\ker(\vartheta_p)$. Then $\mathfrak m$ has residual characteristic $p$ and is Eisenstein because $p\nmid 6MDQ\ell$. Since
\[ \ker(\vartheta_p)\subset J_1^D(MQ)[p]\oplus J_1^D(MQ)[p], \]
it follows that $\mathfrak m$ belongs to the support of $J_1^D(MQ)[p]$. As $p\nmid C_3$, the ideal $\mathfrak m$ belongs to the support of the $\mathbb T$-module $A_i[p]$ for some $i\in\{1,\dots,m\}$. Next, since $p\nmid C_4$, the isogeny $\psi_i:A_i\rightarrow A_{f_{0,i}}$ induces an isomorphism $A_i[p]\simeq A_{f_{0,i}}[p]$ of $G_\Q$-modules where, as before, $f_{0,i}$ is the classical cusp form associated with $f_i$ by the Jacquet--Langlands correspondence. Hence $\mathfrak m$ belongs to the support of the $\T$-module $A_{f_{0,i}}[p]$ as well. But, as pointed out before, $\mathfrak m$ is Eisenstein, so the $G_\Q$-representation $A_{f_{0,i}}[p]$ is reducible, and this proves that $p|C_5$.

The second (and final) step is an easy diagram chasing. Suppose that $p\nmid 6C_3C_4C_5MDQ\ell$. Thanks to the first step, we already know that $\vartheta_p$ is injective (note that the order of $\ker(\vartheta_p)$ is \emph{a priori} a power of $p$). The commutativity of diagram \eqref{diagram-X} shows that
\[ \ker(\xi_p)\subset\ker\bigl((\eta\oplus\eta)\circ(\sigma\oplus\sigma)\bigr), \]
so the order of $\ker(\xi_p)$ divides $C_1C_2$, whence $p|C_1C_2$. \end{proof}

\section{Darmon points on Jacobians of Shimura curves} \label{sec-das}

In this section assume that $D>1$. Our goal is to define Darmon points on Jacobians of Shimura curves over $\Q$ and on closely related abelian varieties. These points are lifts of the local points on elliptic curves introduced by M. Greenberg in \cite{Gr}. The constructions we perform and the conjectures we formulate are the counterparts of those proposed by Dasgupta in \cite[\S 3.3]{Das} when $D=1$, later conjecturally refined by Bertolini, Darmon and Dasgupta in \cite[\S\S 1.2--1.3]{BDD}. We keep the notation of Section \ref{ihara-sec} in force for the rest of the article.

\subsection{Rigid uniformizations of Jacobians of Shimura curves} \label{rig-subs}

In this subsection we recall, and conjecturally refine, the main results of \cite{LRV}.

Denote by $H$ the maximal torsion-free quotient of the cokernel of the map $\pi^*$ introduced in \S \ref{finite-subsec}, let $J_0^D(M\ell)$ be the Jacobian variety of $X_0^D(M\ell)$ and let $J^D_0(M\ell)^{\text{$\ell$-new}}$ be its $\ell$-new quotient, whose dimension will be denoted by $g$; the abelian group $H$ is free of rank $2g$. Now consider the torus
\[ T:=\mathbb G_m\otimes_\Z H \]
where $\mathbb G_m$ denotes the multiplicative group (viewed as a functor on commutative $\Q$-algebras). We will regard $H$ and $T$ as $\Gamma_\ell$-modules with trivial action, where $\Gamma_\ell$ is the Ihara group of \S \ref{ihara-subsec}. Write $\mathcal M_0(H)$ for the $\G_\ell$-module of $H$-valued measures on $\PP^1(\Q_\ell)$ with total mass $0$.

In analogy with what is proved by Dasgupta in \cite{Das} for modular Jacobians, the abelian variety $J^D_0(M\ell)^{\text{$\ell$-new}}$ is uniformized by means of a suitable quotient of $T$. In order to do this, in \cite[Sections 4--6]{LRV} an explicit element $\boldsymbol\mu$ in the cohomology group $H^1\bigl(\G_\ell,\mathcal M_0(H)\bigr)$ is introduced as follows.

Denote by $\mathcal{T}$ the Bruhat--Tits tree of $\PGL_2(\Q_\ell)$, by $\mathcal{V}$ the set of its vertices and by $\mathcal{E}$ the set of its oriented edges. For any edge $e\in\E$ write $s(e),t(e)\in\V$ for its source and its target, respectively, and $\bar e$ for the same edge with reversed orientation. Let $v_\ast$ be the distinguished vertex corresponding to the maximal order $\M_2(\Z_\ell)$ and let $e_\ast$ be the edge emanating from $v_\ast$ and corresponding to the Eichler order consisting of the matrices $\smallmat abcd\in\M_2(\Z_\ell)$ such that $\ell|c$. Set $\hat{v}_\ast:=t(e_\ast)$.

For any abelian group $M$ let $\cF(\V,M)$ and $\cF(\E,M)$ denote the set of maps $m:\V\ra M$ (respectively, $m:\E\ra M$). Both are natural left $\Gamma_\ell$-modules with action $(\g\cdot m)(x):=m(\g^{-1}x)$ for any $\g\in\G_\ell$ and $x\in\V$ or $\E$. Define also
\[ \cF_0(\E,M):=\bigl\{m\in\cF(\E,M)\mid m(\bar e)=-m(e)\bigr\} \]
and
\[ \cF_{\rm har}(M):=\Bigg\{m\in\cF_0(\E,M)\;\bigg|\;\text{$\sum_{s(e)=v}m(e)=0$ for all $v\in\V$}\Bigg\}, \]
which are $\G_\ell$-submodules of $\cF(\E,M)$. The $\cF_{\rm har}(H)$ can be identified with $\mathcal M_0(H)$.

Fix once and for all
\begin{itemize}
\item a prime number $r\nmid\ell DM$;
\item a system of representatives $\{g_i\}_{i=0}^\ell$ for $\Gap\backslash\Ga$;
\item a system of representatives $\mathcal Y=\{\g_e\}_{e\in\E^+}$ for $\Gap\backslash\G_\ell$ such that $\g_e(e)=e_\ast$ and of the form
\[ \g_e=g_{i_1}\hat{g}_{j_1}g_{i_2}\hat{g}_{j_2}\cdots g_{i_s}\hat{g}_{j_s}\qquad\text{with $i_k,j_k\in\{0,\dots,\ell\}$} \]
for every \emph{even} oriented edge $e\in\E^+$.
\end{itemize}
Notice that, with these choices, for every even vertex $v\in\V^+$ there exists an edge $e_0$ with $s(e_0)=v$ such that, putting $\g_v:=\g_{e_0}$, we have $\{\g_e\}_{s(e)=v}=\{g_i\g_v\}_{i=0}^\ell$. This way, the set $\{\g_v\}_{v\in\V^+}$ is also a system of representatives for $\Ga\backslash\G_\ell$ satisfying $\g_v(v)=v_\ast$ for every $v\in\V^+$. Similarly, for any odd vertex $v\in\V^-$ we have $\{\g_e\}_{t(e)=v}=\{\hat{g}_i\g_v\}_{i=0}^\ell$ where $\{\g_v\}_{v\in\V^-}$ is a system of representatives for $\hGa\backslash\G_\ell$ satisfying $\g_v(v)=\hat{v}_\ast$ for every $v\in\V^-$.

The next object made its first appearance in \cite[\S 4]{LRV}, where it is shown that it is indeed well defined. The reader is referred to \cite[Section 2]{LRV} for a discussion of Hecke operators on group homology and cohomology.

\begin{definition} \label{mu}
Set
\[ \boldsymbol\mu:=(T_r-r-1)\cdot\boldsymbol\mu^{\mathcal Y}\in H^1\bigl(\G_\ell,\cF_{\rm har}(H)\bigr)=H^1\bigl(\G_\ell,\mathcal M_0(H)\bigr) \]
where $T_r$ is the $r$-th Hecke operator and $\boldsymbol\mu^{\mathcal Y}$ is the class of the cocycle
\[ \mu^{\mathcal Y}\in Z^1\bigl(\G_\ell,\mathcal M_0(H)\bigr),\qquad\mu^{\mathcal Y}_\g(U_e):=[g_{\g,e}]\quad\text{for all $\g\in\G_\ell$ and $e\in\E^+$}. \]
Here $g_{\g,e}:=\g_e\g\g^{-1}_{\g^{-1}(e)}\in\Gap$ and for every $g\in\Gap$ we write $[g]\in H$ for the class of $g$ in the quotient $H$ of $H_1\bigl(\Gap,\Z\bigr)\simeq\Gap^{\rm ab}$. Finally, $U_e:=\g_e^{-1}(\Z_\ell)$.
\end{definition}

Now we briefly recall from \cite[\S 5.1]{LRV} the integration pairing between $\Div^0\mathcal H_\ell$ and $\mathcal M_0(H)$. For any $d\in\Div^0\mathcal H_\ell$ let $f_d$ denote a rational function on $\PP^1(\C_\ell)$ such that $\mathrm{div}(f_d)=d$. The function $f_d$ is well defined only up to multiplication by constant non-zero functions; however, since the multiplicative integral of a non-zero constant against a measure $\nu\in\mathcal M_0(H)$ is $1$, we get a $\GL_2(\Q_\ell)$-invariant pairing 
\begin{equation} \label{pairing-3}
\Div^0\mathcal H_\ell\times\mathcal M_0(H)\longrightarrow T(\C_\ell),\qquad(d,\nu):=\mint_{\PP^1(\Q_\ell)}f_d\,d\nu.
\end{equation}
We refer the reader to \cite[\S 5.1]{LRV} for the definition of the multiplicative integral as a limit of Riemann products. Finally, by cap product we obtain from \eqref{pairing-3} a pairing
\[ H_1\bigl(\G_\ell,\Div^0\mathcal H_\ell\bigr)\times H^1\bigl(\Gamma_\ell,\mathcal M_0(H)\bigr)\longrightarrow T(\C_\ell). \]
It follows that the cohomology class $\boldsymbol\mu$ defines an integration map on the homology group $H_1(\G_\ell,\Div^0\mathcal H_\ell)$ with values in $T(\C_\ell)$.

Composing the boundary homomorphism $H_2(\G_\ell,\Z)\rightarrow H_1(\G_\ell,\Div^0\mathcal H_\ell)$ induced by the degree map with the integration map produces a further map $H_2(\G_\ell,\Z)\rightarrow T(\C_\ell)$ whose image is denoted by $L$. It turns out that $L$ is a lattice of rank $2g$ in $T(\Q_\ell)$ which is preserved by the action of a suitable Hecke algebra. Finally, let $K_\ell$ denote the (unique, up to isomorphism) unramified quadratic extension of $\Q_\ell$.

The following is \cite[Theorem 1.1]{LRV}.

\begin{theorem} \label{GreenbergConj2}
The quotient $T/L$ admits a Hecke-equivariant isogeny over $K_\ell$ to the rigid analytic space associated with the product of two copies of $J_0^D(M\ell)^{\text{$\ell$-{\rm new}}}$.
\end{theorem}

In fact, something more precise can be said. Write $W_\infty$ for the Atkin--Lehner involution defined in \cite[\S 2.2]{LRV}, and for any $\Z[W_\infty]$-module $M$ and sign $\epsilon\in\{\pm\}$ set $M_\epsilon:=M/(W_\infty-\epsilon1)$. Define
\[ T_\epsilon:=\mathbb G_m\otimes_\Z H_\epsilon. \]
Since the cokernel of the canonical map $H\rightarrow H_+\oplus H_-$ is supported at $2$, it follows that there exists an isogeny of $2$-power degree
\begin{equation} \label{T/L-isogeny-eq}
T/L\longrightarrow T_+/L_+\oplus T_-/L_-
\end{equation}
of rigid analytic tori over $\Q_\ell$. Then Theorem \ref{GreenbergConj2} is proved in \cite{LRV} by showing that for all $\epsilon\in\{+,-\}$ the quotient $T_\epsilon/L_\epsilon$ admits a Hecke-equivariant isogeny over $K_{\ell}$ to the rigid analytic space associated with the abelian variety $J_0^D(M\ell)^{\text{$\ell$-{\rm new}}}$. In the sequel we shall assume the following variant of \cite[Conjecture 1.5]{BDD}.

\begin{conjecture} \label{rigid-analytic-conjecture}
If $\epsilon\in\{+,-\}$ then the quotient $T_\epsilon/L_\epsilon$ is isomorphic over $K_\ell$ to the rigid analytic space associated with an abelian variety $J^{(\ell)}_\epsilon$ defined over $\Q$.
\end{conjecture}

As in \emph{loc. cit.}, we expect that the abelian variety $J_\epsilon^{(\ell)}$ will be endowed with a natural action of the Hecke algebra and that the isomorphism of Conjecture \ref{rigid-analytic-conjecture} will be Hecke equivariant; moreover, we also expect that if one lets the non-trivial element of $\Gal(K_\ell/\Q_\ell)$ act on $T/L$ via the Hecke operator $U_\ell$ then the above isomorphism will be defined over $\Q_\ell$. Granting Conjecture \ref{rigid-analytic-conjecture}, fix once and for all isomorphisms
\begin{equation} \label{rigid-isom-eq}
T_\pm/L_\pm\overset\simeq\longrightarrow J^{(\ell)}_\pm
\end{equation}
defined over $K_\ell$.

\subsection{Darmon points on $J^{(\ell)}_\pm$ and on $J_0^D(M\ell)^{\text{$\ell$-{\rm new}}}$} \label{SH-subsection}

In this subsection we also assume that $\ell$ is inert in $K$, so $K_\ell$ is nothing other than the completion of $K$ at the prime above $\ell$. We freely use the notation of \cite{LRV}, to which we refer for all details. Since $\ell$ is kept fixed in the discussion to follow, for simplicity we set
\[ \Gamma:=\Gamma_\ell. \]
In \cite[\S 7.3]{LRV} a class $\boldsymbol{d}\in H^2\bigl(\Gamma,T(\C_\ell)\bigr)$ is introduced whose image in $H^2\bigl(\Gamma,T(\C_\ell)/L\bigr)$ is trivial; moreover, the lattice $L$ is the smallest subgroup of $T(\Q_\ell)$ with this property. Choose a representative $\mu$ of $\boldsymbol\mu$. If $z\in K_\ell-\Q_\ell$ then $\boldsymbol{d}$ can be represented by the $2$-cocycle $d=d_z\in Z^2\bigl(\Gamma,T(K_\ell)\bigr)$ given by
\begin{equation} \label{desc-d}
(\gamma_1,\gamma_2)\longmapsto\mint_{\PP^1(\Q_\ell)}\frac{t-\gamma^{-1}_1(z)}{t-z}d\mu_{\gamma_2}(t).
\end{equation}
It follows that there exists a map $\beta=\beta_z:\Gamma\rightarrow T/L$ such that
\begin{equation} \label{def-beta}
\beta_{\gamma_1\gamma_2}\cdot\beta_{\gamma_1}^{-1}\cdot\beta_{\gamma_2}^{-1}\equiv d_{\gamma_1,\gamma_2}\pmod{L}
\end{equation}
for all $\gamma_1,\gamma_2\in\Gamma$. Notice that $\beta$ is well defined only up to elements of $\Hom(\Gamma,T/L)$.

Denote $\vartheta:K\hookrightarrow\R$ the embedding fixed at the beginning of this paper and choose also an embedding $K\hookrightarrow\C_{\ell}$. If $\cO$ is an order of $K$ then an embedding $\psi:K\hookrightarrow B$ is said to be an \emph{optimal embedding of $\cO$ into $R$} if $\psi^{-1}(R)=\cO$. Denote $\text{Emb}(\cO,R)$ the set of such embeddings. For every $\psi\in\text{Emb}(\cO,R)$ there is a unique $z_\psi\in K_\ell-\Q_\ell$ such that
\[ \psi(\alpha)\binom{z_\psi}{1}=\alpha\binom{z_\psi}{1}\qquad\text{for all $\alpha\in K$}. \]
By Dirichlet's unit theorem, the abelian group of units in $\cO$ of norm $1$ is free of rank $1$; let $\varepsilon$ be the generator of this group such that $\vartheta(\varepsilon)>1$, then set $\gamma_\psi:=\psi(\varepsilon)\in\Gamma_0^D(M)$.

Let $t=t_\ell$ denote the exponent of $\Gamma^{\rm ab}$. Set
\[ \mathcal P_\psi:=t\cdot\beta_{z_\psi}(\gamma_\psi)\in T(K_\ell)/L. \]
Multiplication by $t$ ensures that $\mathcal P_\psi$ does not depend on the choice of a map $\beta$ as above. 

\begin{proposition} \label{lemma2.2}
The point $\mathcal P_\psi$ does not depend on the choice of a representative of $\boldsymbol\mu$. 
\end{proposition}

\begin{proof} Let $\mu$ and $\mu'$ be two representatives for $\boldsymbol\mu$. It turns out that the $2$-cocycles $d_{z_\psi}$ and $d'_{z_\psi}$ defined as in \eqref{desc-d} in terms of $\mu$ and $\mu'$, respectively, are cohomologous. More precisely, there exists a map $\nu:\Gamma\rightarrow T(K_p)$ such that 
\[ d_{z_\psi}(\g_1,\g_2)=d'_{z_\psi}(\g_1,\g_2)\cdot\nu(\g_1\g_2)\cdot\nu(\g_1)^{-1}\cdot\nu(\g_2)^{-1} \]
for all $\g_1,\g_2\in\Gamma$. One can explicitly write an expression for $\nu$ as follows. Let $m\in\mathcal M_0(H)$ be such that $\mu_\g=\mu'_\g+\g(m)-m$ for all $\g\in\Gamma$; a direct computation shows that 
\begin{equation} \label{eq-k}
\nu(\g)=\bigg(\mint_{\PP^1(\Q_p)}\frac{s-\gamma^{-1}(z_\psi)}{s-z_\psi}dm(s)\bigg)\cdot\varphi(\g)
\end{equation} 
where $\varphi:\Gamma\rightarrow T(K_p)$ is a homomorphism. Write $\bar\nu$ for the composition of $\nu$ with the projection onto $T(K_p)/L$. If $\beta_{z_\psi}:\Gamma\rightarrow T/L$ (respectively, $\beta'_{z_\psi}:\Gamma\rightarrow T/L$) splits $d_{z_\psi}$ (respectively, $d'_{z_\psi}$) modulo $L$ then $\beta_{z_\psi}=\beta'_{z_\psi}\cdot\bar\nu\cdot\varphi'$ for a suitable homomorphism $\varphi':\Gamma\rightarrow T/L$. It follows that 
\begin{equation} \label{eq--}
t\cdot\beta_{z_\psi}=\bigl(t\cdot\beta'_{z_\psi}\bigr)\cdot(t\cdot\bar\nu).
\end{equation} 
Since $\gamma_\psi(z_\psi)=z_\psi$ and $m$ has total mass $0$, equation \eqref{eq-k} shows that $t\cdot\nu(\gamma_\psi)=1$. By definition of the point $\mathcal P_\psi$, the result follows from this and equation \eqref{eq--}. \end{proof}

The next proposition studies the dependence of $\mathcal P_\psi$ on $\psi$.
 
\begin{proposition}
The point $\mathcal P_\psi$ depends only on the $\Gamma_0^D(M)$-conjugacy class of $\psi$. 
\end{proposition}

\begin{proof} Fix an embedding $\psi$, an element $\gamma\in\Gamma_0^D(M)$ and set $\psi':=\gamma\psi\gamma^{-1}$. As in \cite[\S 4.2]{LRV}, choose a radial (in the sense of \cite[Definition 4.7]{LRV}) system $\mathcal Y_{\rm rad}=\{\gamma_e\}_{e\in\mathcal E^+}$ to compute $\boldsymbol\mu$, and introduce the set 
\[ \mathcal Y'_{\rm rad}:=\bigl\{\gamma'_e:=\gamma\gamma_{\gamma^{-1}(e)}\gamma^{-1}\bigr\}_{e\in\mathcal E^+}. \] 
One checks that $\mathcal Y_{\rm rad}'$ is again a radial system. A simple computation shows that 
\begin{equation} \label{eq-prop2.4}
\mint_{\PP^1(\Q_p)}\frac{t-\gamma_1^{-1}z_\psi}{t-z_\psi}d\mu^{\mathcal Y_{\rm rad}}_{\gamma_2}(t)=\mint_{\PP^1(\Q_p)}\frac{t-\gamma\gamma_1^{-1}\gamma^{-1}z_{\psi'}}{t-z_{\psi'}}d\mu^{\mathcal Y_{\rm rad}'}_{\gamma\gamma_2\gamma^{-1}}(t).
\end{equation}
If $\beta'_{z_{\psi'}}$ splits the $2$-cocycle 
\[ (\gamma_1,\gamma_2)\longmapsto\mint_{\PP^1(\Q_p)}\frac{t-\gamma_1^{-1}z_{\psi'}}{t-z_{\psi'}}d\mu^{\mathcal Y_{\rm rad}'}_{\gamma_2}(t) \] 
then equation \eqref{eq-prop2.4} ensures that for all $\tilde\gamma\in\Gamma$ we can take $\beta_{z_\psi}(\tilde\gamma)=\beta'_{z_{\psi'}}(\gamma\tilde\gamma\gamma^{-1})$. By Proposition \ref{lemma2.2}, the point $\mathcal P_{\psi'}$ does not depend on the choice of the representative for $\boldsymbol\mu$. Since $\mu^{\mathcal Y_{\rm rad}'}$ is a representative of $\boldsymbol\mu$ by \cite[Lemma 4.11]{LRV}, it follows that 
\[ \mathcal P_{\psi'}=t\cdot\beta'_{z_{\psi'}}(\gamma_{\psi'})=t\cdot\beta_{z_\psi}(\gamma_\psi)=\mathcal P_\psi, \] 
as was to be shown. \end{proof}

Although, in light of this result, the symbol $\mathcal P_{[\psi]}$ would be more appropriate, for notational simplicity we will continue to write $\mathcal P_\psi$ for the points we have just introduced. However, the reader should always keep in mind that $\mathcal P_\psi=\mathcal P_{\psi'}$ whenever $\psi$ and $\psi'$ are $\Gamma_0^D(M)$-conjugate.

Now let $\nu_\pm:T/L\rightarrow J^{(\ell)}_\pm$ be the two maps obtained by composing isogeny \eqref{T/L-isogeny-eq} with the canonical projections onto the factors and then with isomorphisms \eqref{rigid-isom-eq}.

\begin{definition} \label{SH-J-dfn}
The \emph{Darmon points on $J^{(\ell)}_\pm$ attached to $\cO$} are the points
\[ P_\psi^\pm:=\nu_\pm(\mathcal P_\psi)\in J_\pm^{(\ell)}(K_\ell) \]
for $\psi\in\Emb(\cO,R)$.
\end{definition}

When a choice of sign $\epsilon\in\{\pm\}$ has been made the point $P_\psi^\epsilon$ will be denoted simply by $P_\psi$ (or even by $P_d$ where $d$ is the conductor of $\cO$, if the embedding $\psi$ is understood). Although in this article we shall ultimately work with points on $J^{(\ell)}_\epsilon$ for a fixed choice of sign $\epsilon$, it is worthwhile to explicitly introduce Darmon points on Jacobians of Shimura curves. To do this, choose isogenies
\begin{equation} \label{T/L-sign-isogeny-eq}
T_\pm/L_\pm\longrightarrow J_0^D(M\ell)^{\text{$\ell$-{\rm new}}}
\end{equation}
over $K_\ell$ and write $\lambda_\pm:T/L\rightarrow J_0^D(M\ell)^{\text{$\ell$-{\rm new}}}$ for the two maps obtained by composing isogeny \eqref{T/L-isogeny-eq} with the canonical projections onto the factors and then with isogenies \eqref{T/L-sign-isogeny-eq}.

\begin{definition} \label{SH-Jacobian-dfn}
The \emph{Darmon points on $J_0^D(M\ell)^{\text{$\ell$-{\rm new}}}$ attached to $\cO$} are the points
\[ \lambda_\pm(\mathcal P_\psi)\in J_0^D(M\ell)^{\text{$\ell$-{\rm new}}}(K_\ell) \]
for $\psi\in\Emb(\cO,R)$.
\end{definition}

If $A$ is an elliptic curve over $\Q$ of conductor $DM\ell$ then the points introduced in Definition \ref{SH-Jacobian-dfn} map to the local points on $A$ defined by M. Greenberg in \cite{Gr} under the modular projection $J_0^D(M\ell)^{\text{$\ell$-{\rm new}}}\rightarrow A$.

We conclude this subsection by stating the algebraicity properties conjecturally satisfied by our Darmon points. Write $H$ for the narrow ring class field of $K$ attached to $\cO$ and denote
\[ (\fr a,\psi)\longmapsto\psi^\fr a \]
the action of $\fr a\in\Pic^+(\cO)$ on ($\Gamma_0^D(M)$-conjugacy classes of) embeddings in $\Emb(\cO,R)$ as described, e.g., in \cite[Ch. III, \S 5C]{Vi} (see also Proposition \ref{emb-bijection-prop}). Finally, let $\Pic^+(\cO)$ be the narrow class group of $\cO$ and let
\[ \text{rec}:\Pic^+(\cO)\overset{\simeq}{\longrightarrow}\Gal(H/K) \]
be the isomorphism induced by the reciprocity map of global class field theory.

For the purposes of the present paper, we formulate our rationality conjecture only for Darmon points on $J^{(\ell)}_\pm$, but completely analogous statements could be given for points on $J_0^D(M\ell)^{\text{$\ell$-{\rm new}}}$ as well.

\begin{conjecture} \label{SH-conjecture}
If $\psi\in\Emb(\cO,R)$ then $P_\psi^\pm\in J^{(\ell)}_\pm(H)$ and
\[ P_{\psi^{\fr a}}^\pm=\mathrm{rec}(\fr a)^{-1}\bigl(P_\psi^\pm\bigr) \]
for all $\fr a\in\Pic^+(\cO)$.
\end{conjecture}

This is the analogue of \cite[Conjecture 1.7]{BDD} and is a refinement of \cite[Conjecture 3.9]{Das}, which in turn is the counterpart of \cite[Conjectures 5.6 and 5.9]{Dar1}. 

\section{Algebraic parts of special values and a theorem of Popa} \label{Popa-section}

Let $E_{/\Q}$ be an elliptic curve of conductor $N$ and let $K$ be a real quadratic field as in the introduction; moreover, again with the notation of the introduction, set
\[ D:=\prod_{q\in\Sigma}q\geq1,\qquad M:=N/D. \]
Let $f$ denote the modular form on $\Gamma_0^D(M)$ (well defined up to scalars) associated with $f_0$ by the Jacquet--Langlands correspondence; in particular, if $D=1$ then $f=f_0$. In this section we introduce the algebraic part of the special value at $s=1$ of the $L$-function
\[ L_K(E,\chi,s)=L_K(f_0,\chi,s)=L_K(f,\chi,s) \]
and describe some consequences of a formula proved by Popa in \cite{Po}.

\subsection{Review of the group structure of $\Pic^+(\cl O_c)$} \label{Pic-subsec}

Recall the notation of the introduction; in particular, let $c\geq1$ be an integer prime to $\delta_KN$. As before, the reciprocity map of global class field theory provides a canonical isomorphism
\[ \text{rec}:\Pic^+(\cl O_c)\overset{\simeq}{\longrightarrow}G_c \]
where $G_c$ is the Galois group over $K$ of the narrow ring class field of $K$ of conductor $c$. Let now $\Pic(\cl O_c)$ be the Picard group of $\cl O_c$, that is the group of homothety classes of proper $\cO_c$-ideals of $K$; class field theory then identifies $\Pic(\cO_c)$ with the Galois group over $K$ of the (weak) ring class field $K_c$ of $K$ of conductor $c$. It turns out that if $h(c)$ is the order of $\Pic(\cO_c)$ and $h^+(c)$ is the order of $\Pic^+(\cO_c)$ then $h^+(c)/h(c)=1$ or $2$, so $H_c$ is an extension of $K_c$ of degree at most $2$ (see, e.g., \cite[Ch. 15, \S I]{Cohn}).

Since $(c,\delta_K)=1$ by assumption, the principal ideal $(\sqrt{\delta_K})$ is a proper $\cl O_c$-ideal of $K$, so we can consider its class $\mathfrak D_K$ in $\Pic^+(\cl O_c)$. Of course, $\mathfrak D_K^2=1$, hence $\mathfrak D_K$ is either trivial or of order $2$. Furthermore, there is a short exact sequence
\begin{equation} \label{pic-short-eq}
0\longrightarrow\{1,\mathfrak D_K\}\longrightarrow\Pic^+(\cl O_c)\longrightarrow\Pic(\cl O_c)\longrightarrow0,
\end{equation}
so the natural surjection $\Pic^+(\cl O_c)\twoheadrightarrow\Pic(\cl O_c)$ is an isomorphism (i.e., $h^+(c)=h(c)$) precisely when $\mathfrak D_K$ is trivial. Equivalently, $\Pic^+(\cl O_c)=\Pic(\cl O_c)$ if and only if the order $\cl O_c$ has a unit of norm $-1$. In general, sequence \eqref{pic-short-eq} does not split; in fact, it splits if and only if the integer $\delta_K$ is not a sum of two squares (see \cite[Ch. 14, \S B]{Cohn}).

Now define the Galois element
\[ \sigma_K:=\text{rec}(\mathfrak D_K)\in G_c. \]
It follows that $\sigma_K$ is trivial when $h^+(c)=h(c)$ and has order $2$ otherwise.

The automorphism $\sigma_K$ plays a special role in our considerations because it allows us to introduce, as in \cite{BD-Duke}, a natural notion of parity for characters of $G_c$. As before, write $\widehat G_c$ for the group of complex-valued characters of $G_c$.

\begin{definition} \label{parity-dfn}
A character $\chi\in\widehat G_c$ is said to be \emph{even} (respectively, \emph{odd}) if $\chi(\sigma_K)=1$ (respectively, $\chi(\sigma_K)=-1$).
\end{definition}

Equivalently, a character is even if it factors through $\Gal(K_c/K)$, and is odd otherwise. In particular, if $h^+(c)=h(c)$ then $\sigma_K=1$ and all characters of $G_c$ are even.

\subsection{Oriented optimal embeddings} \label{oriented-subsec}

Equip $R$ and $\cl O_c$ with local orientations at prime numbers dividing $N=DM$, i.e., ring homomorphisms
\[ \fr O_q:R\longrightarrow k_q,\qquad\fr o_q:\cl O_c\longrightarrow k_q \]
for every prime $q|N$ where $k_q$ stands for the finite field with $q$ (respectively, $q^2$) elements if $q|M$ (respectively, $q|D$).

Write $\text{Emb}(K,B)$ for the set of embeddings of $K$ into $B$, which is non-empty because all the primes at which $B$ is ramified are inert in $K$. The group $B^\times$ acts on $\text{Emb}(K,B)$ by conjugation on $B$ and the stabilizer of $\psi\in\text{Emb}(K,B)$ is the (non-split) torus $\psi(K^\times)$. We say that $\psi\in\text{Emb}(K,B)$ is an \emph{oriented optimal embedding of $\cO_c$ into $R$} if $\psi\in\text{Emb}(\cO_c,R)$ and
\[ \fr O_q\circ\psi|_{\cl O_c}=\fr o_q \]
for every prime $q|N$. The set of all such embeddings will be denoted by $\E(\cO_c,R)$, and the cardinality of the set of $\Gamma_0^D(M)$-conjugacy classes of elements of $\E(\cl O_c,R)$ is $h^+(c)$.

Let $\omega_\infty\in R^\times$ be an element of reduced norm $-1$. Note that $\omega_\infty$ normalizes $\Gamma_0^D(M)$; in fact, all such elements lie in a single orbit for the action of $\Gamma_0^D(M)$. For any $\gamma\in B^\times$ set
\begin{equation} \label{gamma-ast-eq}
\gamma^\ast:=\omega_\infty\gamma\omega_\infty^{-1}.
\end{equation}
In particular, $\gamma^\ast\in R$ when $\gamma\in R$. Moreover, if $\psi\in\E(\cl O_c,R)$ then it is immediate to check that
\[ \psi^\ast:=\omega_\infty\psi\omega_\infty^{-1} \]
is in $\E(\cl O_c,R)$ too. By definition, if $\psi(\sqrt{\delta_K})=\gamma$ then $\psi^\ast(\sqrt{\delta_K})=\gamma^\ast$.

\begin{proposition} \label{emb-bijection-prop}
There exists a bijection
\[ F:\E(\cl O_c,R)/\Gamma_0^D(M)\longrightarrow\Pic^+(\cl O_c) \]
such that $F([\psi^\ast])=\mathfrak D_K\cdot F([\psi])$ for all $\psi\in\E(\cl O_c,R)$.
\end{proposition}

\begin{proof} To begin with, the claimed correspondence is not canonical, as $\E(\cl O_c,R)/\Gamma_0^D(M)$ is naturally a torsor under the action of $\Pic^+(\cl O_c)$. In order to describe it, we are thus led to fix an auxiliary optimal embedding $\psi_0\in\E(\cl O_c,R)$. Now we can provide an explicit bijection
\begin{equation} \label{inv-bijection-eq}
\Pic^+(\cl O_c)\lra\E(\cl O_c,R)/\Gamma_0^D(M)
\end{equation}
as follows. Given the class $[\mathfrak a]\in\Pic^+(\cl O_c)$ of an ideal $\mathfrak a$, the set $R\psi_0(\mathfrak a)$ is a left ideal, which is known to be principal because $B$ is indefinite. Since $n(R^\times)=\{\pm1\}$, we may find an element $a\in R$ with reduced norm $n(a)>0$ such that $R\psi_0(\mathfrak a)=Ra$, this $a$ being well defined up to elements in $\Gamma_0^D(M)$. Set
\[ \psi_{[\mathfrak a]}:=a\psi_0 a^{-1}\in\E(\cl O_c,R). \]
It is easy to check that the rule $[\mathfrak a]\mapsto\bigl[\psi_{[\mathfrak a]}\bigr]$ induces a well-defined bijection as in \eqref{inv-bijection-eq}. The inverse of \eqref{inv-bijection-eq} can then be taken to be the searched-for $F$ in the statement of the proposition.

Finally, notice that if $\mathfrak a=\mathfrak b\cdot(\sqrt{\delta_K})$ then we can take
\[ a=\om_\infty\cdot b\cdot\psi_0(\sqrt{d_K}) \]
where $b\in R$ is such that $n(b)>0$ and $R\psi_0(\mathfrak b)=Rb$. Hence
\[ \psi_{[\mathfrak a]}=\bigl(\om_\infty\cdot b\cdot\psi_0(\sqrt{d_K})\bigr)\psi_0\bigl(\psi_0(\sqrt{d_K})^{-1}\cdot b^{-1}\cdot\om_\infty^{-1}\bigr). \]
Since $\psi_0(\sqrt{d_K})\psi_0\psi_0(\sqrt{d_K})^{-1}=\psi_0$ because $\cl O_c$ is a commutative ring, we conclude that
\[ \psi_{[\mathfrak b]\mathfrak D_K}=\psi_{[\mathfrak b]}^\ast. \]
Thus
\[ F([\psi^\ast])=\mathfrak D_K\cdot F([\psi]) \]
for all $\psi\in\E(\cl O_c,R)$, as was to be shown. \end{proof}

We choose once and for all an optimal embedding $\psi_0\in\E(\cl O_c,R)$ and regard the bijection $F$ of Proposition \ref{emb-bijection-prop}, built out of $\psi_0$, as fixed. Notice that, by this proposition, $[\psi^\ast]=[\psi]$ if and only if $h^+(c)=h(c)$. Observe also that this is the case precisely when $\om_\infty$ can be taken to lie in $\cl O_c$. Consider the composition
\[ G:=\text{rec}\circ F:\E(\cl O_c,R)/\Gamma_0^D(M)\longrightarrow G_c, \]
which is a bijection satisfying
\begin{equation} \label{G-ast-eq}
G([\psi^\ast])=\sigma_K\cdot G([\psi])
\end{equation}
for all $\psi\in\E(\cl O_c,R)$. Now for every $\sigma\in G_c$ choose an embedding
\[ \psi_\sigma\in G^{-1}(\sigma), \]
so that the family $\{\psi_\sigma\}_{\sigma\in G_c}$ is a set of representatives of the $\Gamma_0^D(M)$-conjugacy classes of oriented optimal embeddings of $\cl O_c$ into $R$. If $\gamma,\gamma'\in R$ write $\gamma\sim\gamma'$ to indicate that $\gamma$ and $\gamma'$ are in the same $\Gamma_0^D(M)$-conjugacy class, and adopt a
similar notation for (oriented) optimal embeddings of $\cl O_c$ into $R$. Since
\[ G([\psi_\sigma^\ast])=\sigma_K\cdot G([\psi_\sigma])=\sigma_K\sigma \]
by equality \eqref{G-ast-eq}, we deduce that
\begin{equation} \label{psi-ast-equiv-eq}
\psi_\sigma^\ast\sim\psi_{\sigma_K\sigma}
\end{equation}
for all $\sigma\in G_c$.

After choosing a (fundamental) unit $\varepsilon_c$ of $\cl O_c$ of norm $1$, normalized so that $\varepsilon_c>1$ with respect to the fixed real embedding of $K$, define
\begin{equation} \label{gamma-sigma-eq}
\gamma_\sigma:=\psi_\sigma(\varepsilon_c)\in\Gamma_0^D(M)
\end{equation}
for all $\sigma\in G_c$. As an immediate consequence of \eqref{psi-ast-equiv-eq} and \eqref{gamma-sigma-eq}, one has
\begin{equation} \label{gamma-ast-equiv-eq}
\gamma_\sigma^\ast=\psi_\sigma^\ast(\varepsilon_c)\sim\psi_{\sigma_K\sigma}(\varepsilon_c)=\gamma_{\sigma_K\sigma}
\end{equation}
for all $\sigma\in G_c$. This seemingly innocuous conjugacy relation will play a crucial role in the proof of Proposition \ref{eigenspace-prop}.

\subsection{Homology of Shimura curves and complex conjugation}\label{homology-subsec}

Let $\T_M=\T_M^D$ be the algebra of Hecke operators acting on cusp forms of weight $2$ on $\Gamma_0^D(M)$, which is generated over $\Z$ by the Hecke operators $T_\ell$ for primes $\ell\nmid DM$ and $U_q$ for primes $q|M$. The algebra $\T_M$ acts naturally on the (singular) homology group $H_1\bigl(X_0^D(M),\Z\bigr)$. As before, let $a_\ell\in\Z$ be the eigenvalue of $f$ for the action of the Hecke operator $T_\ell$ (respectively, $U_\ell$) if $\ell\nmid M$ (respectively, if $\ell|M$). Set
\[ I_f:=\big\langle T_\ell-a_\ell,\;\ell\nmid D M;\;U_q-a_q,\;q|M\big\rangle\subset\T_M, \]
so that $I_f$ is the kernel of the algebra homomorphism
\begin{equation} \label{kerf}
\varphi_f:\T_M\longrightarrow\Z,\qquad T_\ell\longmapsto a_\ell,\qquad U_q\longmapsto a_q
\end{equation}
determined by $f$. As a piece of notation, for any $\T_M$-module $A$ write $A_f:=A/I_fA$ for the maximal quotient of $A$ on which $\T_M$ acts via $\varphi_f$.

We want to embed $X_0^D(M)$ into its Jacobian. If $D=1$ then let
\begin{equation} \label{param-1-eq}
\zeta:X_0(M)\longrightarrow J_0(M)
\end{equation}
be the usual map sending the cusp $\infty$ on $X_0(M)$ to the origin of $J_0(M)$.

If $D>1$ then, following \cite{zh1}, let the \emph{Hodge class} be the unique $\xi\in\Pic(X_0^D(M))\otimes\Q$ of degree $1$ on which the Hecke operators at primes not dividing $M$ act as multiplication by their degree (see \cite[p. 30]{zh1} for an explicit expression of $\xi$ and \cite[\S 3.5]{CV} for a detailed exposition). Writing $J_0^D(M)$ for the Jacobian variety of $X_0^D(M)$, one can define a map
\[ X_0^D(M)\lra J_0^D(M)\otimes\Q \]
by sending a point $x\in X_0^D(M)$ to the class $[x]-\xi$. Multiplying this map by a suitable integer $m\gg0$ gives a finite embedding
\begin{equation} \label{param-D-eq}
\zeta:X_0^D(M)\lra J_0^D(M)
\end{equation}
defined over $\Q$ (cf. \cite[\S 3.5]{CV}), which we fix once and for all.

Choose a parametrization
\[ J_0^D(M)\longrightarrow E \]
defined over $\Q$, whose existence is guaranteed by the modularity of $E$ and (when $D>1$) the Jacquet--Langlands correspondence. Denote by
\[ \pi_E:X_0^D(M)\longrightarrow E \]
the surjective morphism over $\Q$ obtained by pre-composing the parametrization above with the map $\zeta$ defined either in \eqref{param-1-eq} or in \eqref{param-D-eq}. Let now $d_E$ be the degree of $\pi_E$, and if $T$ is a finite set of prime numbers write $\Z_T$ for the localization of $\Z$ in which the primes in $T$ are inverted. Throughout this article we fix a (minimal) finite set of primes $S$ such that
\begin{itemize}
\item all prime divisors of $6d_E$ belong to $S$;
\item the $\Z_S$-module $H_1\bigl(X_0^D(M),\Z_S\bigr)_f$ is torsion-free.
\end{itemize}
The universal coefficient theorem for homology ensures that this can actually be done. Then push-forward gives an isomorphism
\begin{equation} \label{map1-eq}
\pi_{E,\ast}:H_1\bigl(X_0^D(M),\Z_S\bigr)_f\overset{\simeq}{\longrightarrow}H_1(E,\Z_S).
\end{equation}
\begin{remarkwr} \label{size-S-rem}
Although -- in order to make our choice somewhat more canonical -- the set $S$ is taken to be minimal, enlarging $S$ does not affect the above two properties, and so all statements proved remain valid when $S$ is replaced by any set containing it. This freedom of modifying the size of $S$ will be exploited in the proof of Theorem \ref{cong-thm}.
\end{remarkwr}

Let $\cl H$ be the complex upper half-plane and let $\Pi:\cl H\rightarrow X_0^D(M)$ be the canonical map. For every point $z_0\in\cl H$ there is a group homomorphism
\begin{equation} \label{map2-eq}
\Gamma_0^D(M)\longrightarrow\pi_1\bigl(X_0^D(M),\Pi(z_0)\bigr)
\end{equation}
defined by the following recipe: if $\gamma\in\Gamma_0^D(M)$ and $\alpha:[0,1]\rightarrow\cl H$ is a path from $z_0$ to $\gamma(z_0)$ then the map \eqref{map2-eq} sends $\gamma$ to the (strict) homotopy class of the loop $\Pi\circ\alpha$ around $\Pi(z_0)$. Since $\cl H$ is simply connected, this class does not depend on the choice of $\alpha$.

By Hurewicz's theorem, the abelianization of $\pi_1\bigl(X_0^D(M),\Pi(z_0)\bigr)$ is canonically isomorphic to $H_1\bigl(X_0^D(M),\Z\bigr)$, hence there is a group homomorphism
\[ [\,\cdot\,]:\Gamma_0^D(M)\longrightarrow H_1\bigl(X_0^D(M),\Z_S\bigr) \]
which is independent of the choice of the base point $z_0$ in $\cl H$.

Recall the elements $\gamma_\sigma\in\Gamma_0^D(M)$ with $\sigma\in G_c$ that were introduced in \S \ref{oriented-subsec}. Since the group $H_1\bigl(X_0^D(M),\Z_S\bigr)$ is abelian, for each $\sigma\in G_c$ the homology class $[\gamma_\sigma]$ does not depend on the representative $\psi_\sigma$ of the $\Gamma_0^D(M)$-conjugacy class of (oriented) optimal embeddings in terms of which $\gamma_\sigma$ was defined (cf. equation \eqref{gamma-sigma-eq}).

Let now $\varepsilon\in R^\times$ be a unit of norm $-1$ and let $\tau$ denote the involution on $\cl H$ given by $z\mapsto\varepsilon(\bar z)$ where $\bar z$ is the conjugate of the complex number $z$. Since $\Gamma_0^D(M)$ is a normal subgroup of $R^\times$, the map $\tau$ descends to an involution on $X_0^D(M)$ by the formula
\begin{equation} \label{tau-eq}
\Pi(z)^\tau=\Pi\bigl(\varepsilon(\bar z)\bigr)
\end{equation}
for all $z\in\cl H$; according to Shimura, this action does not depend on the choice of an $\varepsilon$ as above and coincides with the natural action of complex conjugation on the Riemann surface $X_0^D(M)$ (\cite[Proposition 1.3]{Sh}).

The rule \eqref{tau-eq} induces an action of $\tau$ on the homology of $X_0^D(M)$. With notation as in \eqref{gamma-ast-eq}, by definition of the homomorphism $[\,\cdot\,]$, for all $\gamma\in\Gamma_0^D(M)$ one has
\begin{equation} \label{gamma-tau-eq}
[\gamma]^\tau=[\gamma^\ast]
\end{equation}
in $H_1\bigl(X_0^D(M),\Z_S\bigr)$. The involution $\tau$ restricts to a permutation of the subset $\bigl\{[\gamma_\sigma]\bigr\}_{\sigma\in G_c}$; the understanding of this permutation provided by equation \eqref{gamma-ast-equiv-eq} will be crucial for our definition of the algebraic part of $L_K(E,\chi,1)$.

\subsection{The algebraic part} \label{alg-part-subsec}

Here we introduce the algebraic part of the special value of $L_K(E,\chi,s)$ at the critical point $s=1$. Set
\[ I_\chi:=\sum_{\sigma\in G_c}\chi^{-1}(\sigma)[\gamma_\sigma]\in H_1\bigl(X_0^D(M),\Z[\chi]_S\bigr). \]
Since the $[\gamma_\sigma]$ do not depend on $z_0$ in $\cl H$, the cycle $I_\chi$ is independent of $z_0$. 

The next result says that $\tau$ acts either as $+1$ or as $-1$ on $I_\chi$ according to the parity of $\chi$ that was introduced in Definition \ref{parity-dfn}.

\begin{proposition} \label{eigenspace-prop}
The cycle $I_\chi$ lies in the $+1$-eigenspace (respectively, $-1$-eigenspace) for $\tau$ if $\chi$ is even (respectively, odd).
\end{proposition}

\begin{proof} Thanks to equality \eqref{gamma-tau-eq} and the conjugacy relation of equation \eqref{gamma-ast-equiv-eq}, one has
\[ \begin{split}
   I_\chi^\tau&=\sum_{\sigma\in G_c}\chi^{-1}(\sigma)[\gamma_\sigma]^\tau=\sum_{\sigma\in G_c}\chi^{-1}(\sigma)[\gamma_\sigma^\ast]=\sum_{\sigma\in G_c}\chi^{-1}(\sigma)[\gamma_{\sigma_K\sigma}]\\[2mm]
   &=\chi(\sigma_K)\cdot\bigg(\sum_{\sigma\in G_c}\chi^{-1}(\sigma_K\sigma)[\gamma_{\sigma_K\sigma}]\bigg)=\chi(\sigma_K)\cdot\bigg(\sum_{\varsigma\in G_c}\chi^{-1}(\varsigma)[\gamma_\varsigma]\bigg)\\[2mm]
   &=\chi(\sigma_K)I_\chi,
   \end{split} \]
whence the claim. \end{proof}

Consider the push-forward
\[ I_{\chi,E}:=\pi_{E,\ast}(I_\chi)\in H_1(E,\Z[\chi]_S), \]
write $H_1\bigl(X_0^D(M),\Z[\chi]_S\bigr)^\pm$ for the eigenspace of $H_1\bigl(X_0^D(M),\Z[\chi]_S\bigr)$ on which the involution $\tau$ acts as multiplication by $\pm1$, and adopt a similar convention for $H_1(E,\Z[\chi]_S)$. Since the morphism $\pi_E$ is defined over $\Q$, one has
\[ I_\chi\in H_1\bigl(X_0^D(M),\Z[\chi]_S\bigr)^\epsilon\;\Longrightarrow\;I_{\chi,E}\in H_1(E,\Z[\chi]_S)^\epsilon \]
for $\epsilon\in\{+,-\}$. The reader is suggested to compare our homology cycle $I_{\chi,E}$ with the twisted sum of period integrals $I(f,\chi)$ introduced in \cite[p. 191]{BD-Duke}.

Keeping in mind that $H_1(E,\Z)$ identifies with the lattice of periods associated with a Weierstrass equation for $E$, it can be checked that both $H_1(E,\Z[\chi]_S)^+$ and $H_1(E,\Z[\chi]_S)^-$ are free of rank $1$ over $\Z[\chi]_S$; here we fix canonical generators $\alpha_E^+$ and $\alpha_E^-$ of these two eigenspaces
over $\Z[\chi]_S$ as described in \cite[\S 2.2]{MaSw}.

Now suppose that $I_\chi\in H_1\bigl(X_0^D(M),\Z[\chi]_S\bigr)^\epsilon$ with $\epsilon\in\{+,-\}$: by Proposition \ref{eigenspace-prop}, the nature of $\epsilon$ depends on the parity of $\chi$. Let $\cl L_K(E,\chi,1)_S$ be the unique element of $\Z[\chi]_S$ such that the equality
\begin{equation} \label{alg-part-eq}
I_{\chi,E}=\cl L_K(E,\chi,1)_S\cdot\alpha_E^\epsilon
\end{equation}
holds in $H_1(E,\Z[\chi]_S)$.

\begin{definition} \label{alg-part-dfn}
The element $\cl L_K(E,\chi,1)_S\in\Z[\chi]_S$ appearing in \eqref{alg-part-eq} is the \emph{algebraic part} of $L_K(E,\chi,1)$.
\end{definition}

Since the finite set $S$ has been fixed once and for all, from here on we drop the dependence of the algebraic part of $L_K(E,\chi,1)$ on $S$ from the notation and simply write $\cl L_K(E,\chi,1)$ in place of $\cl L_K(E,\chi,1)_S$.

Before we proceed to crucial considerations on the vanishing of $L_K(E,\chi,1)$, a few comments are in order.
\begin{remark}
By construction, $I_\chi$ naturally belongs to the submodule $H_1\bigl(X_0^D(M),\Z[\chi]\bigr)$. In fact, as in \cite{BDD}, the need to localize at $S$ will become evident only later, but for clarity of exposition we decided to introduce the required formalism at the outset of our work.
\end{remark}
\begin{remark}
The definition of the algebraic part of the special value $L_K(E,\chi,1)$ given by Bertolini, Darmon and Dasgupta in \cite{BDD} is slightly different. In fact, $\cl L_K(E,\chi,1)$ is defined in \cite[Section 2]{BDD} to be the natural image of $I_\chi$ in $H_1\bigl(X_0^D(M),\Z[\chi]_S\bigr)_f$ (note, however, that the authors of \emph{loc. cit.} only consider the classical case of modular curves, with $c=1$ and trivial $\chi$). On the other hand, tensoring the isomorphism in \eqref{map1-eq} with $\Z[\chi]_S$ over $\Z_S$ shows that the two definitions of $\cl L_K(E,\chi,1)$ are essentially equivalent.
\end{remark}

\subsection{Vanishing of the special value} \label{popa-subsec}

The goal of this subsection is to prove that the special value of $L_K(E,\chi,s)$ vanishes exactly when its algebraic part does. This is a consequence of a result proved by Popa in \cite[Section 5]{Po} and reformulated in more classical terms in \cite[Section 6]{Po} when $D=1$ and $\chi$ is unramified. In this special case, Popa's computations are based on a very explicit description of a bijection between suitable ideal classes and conjugacy classes of optimal embeddings. While it seems difficult to exhibit such an explicit correspondence when $D>1$, Proposition \ref{emb-bijection-prop} provides sufficient information to allow for a ``classical'' formulation of Popa's theorem in the general setting as well.

The result we are interested in is the following

\begin{theorem}[Popa] \label{popa-thm}
The special value $L_K(E,\chi,1)$ is non-zero if and only if $\cl L_K(E,\chi,1)$ is non-zero.
\end{theorem}

\begin{proof} As already remarked, this is a consequence of the formula for $L_K(E,\chi,1)$ proved by Popa in \cite{Po}. Since the results of Popa are expressed in the adelic language of automorphic representations, we explain how to deduce the theorem in the formulation that is convenient for our purposes. In fact, in equality \eqref{popa-eq5} we give an explicit formula for $L_K(E,\chi,1)$ when the character $\chi$ is not necessarily trivial; in doing this, we freely use the notation of \cite{Po}.

First of all, observe that, due to the normalization commonly adopted in automorphic-theoretic contexts (cf. \cite[\S 5.14]{IK}), the special value of $L_K(E,\chi,s)$ at $s=1$ corresponds to $L(1/2,\pi_f\times\pi_\chi)$ in \cite{Po}. As recalled in \S \ref{oriented-subsec}, the $\Gamma_0^D(M)$-conjugacy classes of oriented optimal embeddings of $\cl O_c$ into $R$ are in bijection with the elements of the Galois group $G_c$. With arguments analogous to those exposed in \cite[Section 6]{Po}, if $\omega_f:=2\pi if(z)dz$ is the differential on $X_0^D(M)$ associated with $f$ one then obtains an equality
\begin{equation} \label{popa-eq1}
|l(\phi_f)|^2=\Bigg|\sum_{\sigma\in G_c}\chi^{-1}(\sigma)\int_{z_0}^{\gamma_\sigma(z_0)}f(z)dz\Bigg|^2=\bigg|\int_{I_\chi}\omega_f\bigg|^2
\end{equation}
where $l$ is a certain linear form on a suitable space of automorphic forms (see \cite[p. 852]{Po}) and $\phi_f$ is the automorphic form attached to $f$ as in \cite[Proposition 5.3.6]{Po}. Equality \eqref{popa-eq1} is the analogue (with $k=1$) of the formula given, in the split case, in \cite[p. 862]{Po} for an unramified $\chi$ (in this setting, see also \cite[Theorem 6.3.1]{Po}, which provides a formulation of Popa's results suitable for the arithmetic applications of \cite{BDD}). Now \cite[Theorem 5.3.9]{Po} with $k=1$ asserts that there is a non-zero constant $\Omega$ (denoted by $C$ in \emph{loc. cit.}) such that
\begin{equation} \label{popa-eq2}
L_K(E,\chi,1)=\frac{\Omega Nc^2}{\sqrt{\delta_K}}\prod_{\ell|Nc}\Big(1+\frac{1}{\ell}\Big)|l(\phi_f)|^2;
\end{equation}
the explicit expression of $\Omega$ in the case where $c=1$ can be found in \cite[\S 5.4]{Po}.

Combining equations \eqref{popa-eq1} and \eqref{popa-eq2} yields immediately the formula
\begin{equation} \label{popa-eq3}
L_K(E,\chi,1)=\frac{\Omega Mc^2}{\sqrt{\delta_K}}\prod_{\ell|Mc}\Big(1+\frac{1}{\ell}\Big)\bigg|\int_{I_\chi}\omega_f\bigg|^2,
\end{equation}
and the claim of the theorem follows from \eqref{popa-eq3} by passing to the push-forward
\[ I_{\chi,E}=\pi_{E,\ast}(I_\chi)\in H_1(E,\Z[\chi]_S). \]
Namely, let $\omega_E$ be a N{\'e}ron differential on the N{\'e}ron model of $E$ over $\Z$; by \cite[Theorem 5.6]{zh2}, there is an equality
\[ \pi_E^\ast(\omega_E)=c(\pi_E)\omega_f \]
with $c(\pi_E)\in\C^\times$; then one has
\begin{equation} \label{popa-eq4}
c(\pi_E)\int_{I_\chi}\omega_f=\int_{I_{\chi,E}}\omega_E=\cl L_K(E,\chi,1)\int_{\alpha_E^\epsilon}\omega_E
\end{equation}
where $\epsilon\in\{+,-\}$ and $I_\chi\in H_1\bigl(X_0^D(M),\Z[\chi]_S\bigr)^\epsilon$. Finally, combining \eqref{popa-eq3} and \eqref{popa-eq4} gives the equality
\begin{equation} \label{popa-eq5}
L_K(E,\chi,1)=\bigl|\cl L_K(E,\chi,1)\bigr|^2\cdot\frac{\Omega Mc^2}{|c(\pi_E)|^2\sqrt{\delta_K}}\prod_{\ell|Mc}\Big(1+\frac{1}{\ell}\Big)\bigg|\int_{\alpha_E^\epsilon}\omega_E\bigg|^2,
\end{equation}
and the theorem is proved. \end{proof}

\section{Admissible primes relative to $f$ and $p$} \label{adm}

For any prime number $q$ fix an isomorphism $E[q]\simeq(\Z/q\Z)^2$ by choosing a basis of $E[q]$ over $\Z/q\Z$ and let
\[ \rho_{E,q}:G_\Q\longrightarrow\GL_2(\Z/q\Z) \]
be the representation of $G_\Q$ acting on $E[q]$.

\subsection{Choice of $p$} \label{choice-p-subsec}

Here we introduce the restrictions on the prime numbers $p$ under which we will prove our main results; they are analogous to those made in \cite[Assumption 4.1]{LV}. Before doing this, recall the finite set of primes $S$ of \S \ref{homology-subsec}, the algebraic part $\cl L_K(E,\chi,1)\in\Z[\chi]_S$ introduced in \S \ref{alg-part-subsec} and the prime $r$ appearing in Definition \ref{mu}. Finally, fix an integer $C$ as in Theorem \ref{Ihara-1}.

\begin{ass} \label{ass}
Suppose that $L_K(E,\chi,1)\neq0$. Then
\begin{enumerate}
\item $p\not\in S$;
\item $p\nmid2cNC\delta_Kh^+(c)(r+1-a_r)$;
\item the Galois representation $\rho_{E,p}$ is surjective;
\item the image of $\cl L_K(E,\chi,1)$ in the quotient $\Z[\chi]_S/p\Z[\chi]_S$ is not zero;
\item $p\nmid|E(H_{c,\fr q})_{\text{tors}}|$ where $H_{c,\fr q}$ is the completion of $H_c$ at a prime $\fr q$ dividing $DM$.
\end{enumerate}
\end{ass}

The ``open image theorem'' of Serre (\cite{Se}) ensures that condition $3$ is satisfied for all but finitely many primes $p$, while the torsion subgroup of $E(H_{c,\fr q})$ is finite by a well-known theorem of Lutz (\cite{Lutz}); moreover, condition $4$ excludes only a finite number of primes $p$ since $\cl L_K(E,\chi,1)\not=0$ by Theorem \ref{popa-thm}. As a consequence, Assumption \ref{ass} is fulfilled by almost all prime numbers $p$. Observe that, in order to avoid ambiguities, the condition $L_K(E,\chi,1)\not=0$ will always explicitly appear in the statements of our results.

\begin{remarkwr}
Condition $5$ in Assumption \ref{ass} is introduced in order to ``trivialize'' the image of the local Kummer map at primes of bad reduction for $E$. The reader is referred to, e.g., \cite{GPa} to see how one could impose suitable local conditions at these primes too. We also expect that Assumption \ref{ass} could be relaxed by using the methods recently proposed by Nekov\'a\v{r} in his work on level raising for Hilbert modular forms of weight two (\cite{Ne}), which greatly improves the techniques introduced in \cite{BD1} and then refined in \cite{LV}. 
\end{remarkwr}

\subsection{Admissible primes} \label{admissible-subsec}

Let $p$ be the prime number chosen in \S \ref{choice-p-subsec} and recall the  quaternionic modular form $f$ of weight $2$ on $\Gamma^D_0(M)$ associated with $E$ by the Jacquet--Langlands correspondence. Following \cite[\S 3.3]{BDD} (see also \cite[\S 2]{BD1} and \cite[\S 4.2]{LV} for an analogous definition in the imaginary quadratic setting), we say that a prime number $\ell$ is \emph{admissible relative to $f$ and $p$} (or $p$-\emph{admissible}, or even simply \emph{admissible}) if it satisfies the following conditions:
\begin{enumerate}
\item $\ell\nmid Npc$;
\item the support of $\Gamma_\ell^{\rm ab}$ is contained in the set of prime divisors of $C\ell$;
\item $\ell$ is inert in $K$;
\item $p\nmid\ell^2-1$;
\item $p|(\ell+1)^2-a_\ell^2$.
\end{enumerate}
Note that, thanks to Theorem \ref{Ihara-1}, the first two conditions exclude only a finite number of primes $\ell$. Moreover, as a consequence of condition $2$ in Assumption \ref{ass}, the prime $p$ does not divide the exponent $t_\ell$ of $\Gamma_\ell^{\rm ab}$ for all admissible primes $\ell$.

For every admissible prime $\ell$ choose once and for all a prime $\lambda_0$ of $H_c$ above $\ell$ (we will never deal with more than one admissible prime at the same time, so ignoring the dependence of $\lambda_0$ on $\ell$ should cause no confusion). Since admissible primes are inert in $K$ and do not divide $c$, if $\ell$ is such a prime then $\ell\cl O_K$ splits completely in $H_c$, hence there are exactly $h^+(c)$ primes of $H_c$ above $\ell$. The choice of $\lambda_0$ allows us to fix an explicit bijection between $G_c$ and the set of these primes via the rule
\begin{equation} \label{bijection-eq}
\sigma\in G_c\longmapsto\sigma(\lambda_0).
\end{equation}
The inverse to this bijection will be denoted
\[ \lambda\longmapsto\sigma_\lambda\in G_c, \]
so that $\sigma_\lambda(\lambda_0)=\lambda$. Finally, an element $\sigma\in G_c$ acts on the group rings $\Z[G_c]$ and $\Z/p\Z[G_c]$ in the natural way by multiplication on group-like elements (that is, $\gamma\mapsto\sigma\gamma$ for all $\gamma\in G_c$).

\begin{lemma} \label{local-iso-lemma}
Let $\ell$ be an admissible  prime relative to $f$ and $p$. The local cohomology groups $H^1_{\rm fin}(H_{c,\ell},E[p])$ and $H^1_{\rm sing}(H_{c,\ell},E[p])$ are both isomorphic to $\Z/p\Z[G_c]$ as $\Z[G_c]$-modules.
\end{lemma}
\begin{proof} Since $p\nmid\ell^2-1$, one can mimic the proof of \cite[Lemma 2.6]{BD1} and show that the groups $H^1_{\rm fin}(K_{\ell},E[p])$ and $H^1_{\rm sing}(K_{\ell},E[p])$ are both isomorphic to $\Z/p\Z$. But the prime ideal $\ell\cl O_K$ of $\cl O_K$ splits completely in $H_c$, hence $H^1_{\rm fin}(H_{c,\ell},E[p])$ and $H^1_{\rm sing}(H_{c,\ell},E[p])$ are both isomorphic to $\Z/p\Z[G_c]$ as $\F_p$-vector spaces. Finally, bijection \eqref{bijection-eq} establishes isomorphisms which are obviously $G_c$-equivariant. \end{proof}

For $\star\in\{\text{fin, sing}\}$ we fix once and for all isomorphisms
\[ H^1_\star(K_\ell,E[p])\simeq\Z/p\Z \]
which will often be viewed as identifications according to convenience.

The next result is the counterpart of \cite[Proposition 4.5]{LV}. In fact, since the group $\Gal(H_c/\Q)$ is generalized dihedral, with the non-trivial element $\rho$ of $\Gal(K/\Q)$ acting on the abelian normal subgroup $G_c$ by
\[ \sigma\longmapsto\rho\sigma\rho^{-1}=\sigma^{-1}, \]
the proof of \cite[Proposition 4.5]{LV} is valid \emph{mutatis mutandis} in our present context as well.

\begin{proposition} \label{existence-admissible-primes}
Let $s$ be a non-zero element of $H^1(H_c,E[p])$. For every $\delta\in\{\pm1\}$ there are infinitely many admissible primes $\ell$ such that $p$ divides $a_\ell+\delta(\ell+1)$ and $\res_\ell(s)\neq0$.
\end{proposition}

The existence result of Proposition \ref{existence-admissible-primes} will be crucially exploited in \S \ref{main-subsec} to show the vanishing of Selmer groups which is one of the goals of this paper.

\section{Level raising and Galois representations} \label{raising-sec}

In this section we prove a level raising result modulo $p$ at admissible primes (Theorem \ref{cong-thm}) and an isomorphism between certain Galois representations over $\F_p$ attached to $J_\epsilon^{(\ell)}$ and $E$ (Theorem \ref{teo-rep}).

\subsection{Raising the level in one admissible prime} \label{raising-subsec}

As in Section \ref{sec-das}, fix a prime $\ell\nmid DM$ and a character $\chi\in\widehat G_c$ whose parity is denoted by $\epsilon$. Recall the modular eigenform $f$ for $\Gamma_0^D(M)$ introduced in Section \ref{Popa-section} and the homomorphism $\varphi_f:\T_M\rightarrow\Z$ of \eqref{kerf}. Write
\[ \bar\varphi_f:\T_M\longrightarrow\Z/p\Z \]
for the composition of $\varphi_f$ with the projection $\Z\rightarrow \Z/p\Z$ and denote by $\m_f$ its kernel, so that $\m_f=I_f+(p)$ where $I_f=\ker(\varphi_f)$.

As is well known, $\pi^*$ is injective and this allows us to identify $H_1\bigl(X_0^D(M),\Z_S\bigr)^2$ with the submodule ${\rm im}(\pi^*)$ of $H_1\bigl(X_0^D(M\ell),\Z_S\bigr)$, which is stable under the action of $\T_{M\ell}$; this provides $H_1\bigl(X_0^D(M),\Z_S\bigr)^2$ with a natural structure of $\T_{\ell M}$-module. More precisely, $\pi^*$ is equivariant for the actions of $T_q$, $t_q$ for primes $q\nmid M\ell$ and of $U_q$, $u_q$ for primes $q|M$, while it intertwines the actions of $\smallmat{T_\ell}{-1}\ell0$ on the domain and of $u_\ell$ on the codomain.

Thanks to \cite[Lemma 6.2]{LRV}, the natural inclusion $\ker(\pi_*)\subset H_1\bigl(X^D_0(M),\Z\bigr)_\epsilon^2$ induces an injection $\ker(\pi_*)\hookrightarrow {\rm coker}(\pi^*)$, so we may consider the $\Z$- and $\Z_S$-modules
\[\Phi_\ell:={\rm coker}(\pi^*)/\ker(\pi_*),\qquad\Phi_{\ell,S}:=\Phi_\ell\otimes\Z_S, \]
respectively, which are endowed with canonical structures of $\T_{M\ell}$-modules and, again by \cite[Lemma 6.2]{LRV}, have finite cardinality.

For any abelian group $M$ endowed with an action of the involution $\tau$, let $M_\pm$ denote the maximal quotient of $M$ on which $\tau$ acts as $\pm1$. Since the maps $\pi_1$ and $\pi_2$ of Section \ref{sec-das} are defined over $\Q$, if $\epsilon\in\{+,-\}$ then there are morphisms
\[  \pi^*_\epsilon:H_1\bigl(X^D_0(M),\Z\bigr)_\epsilon^2\rightarrow H_1\bigl(X^D_0(M\ell),\Z\bigr)_\epsilon,\quad \pi_{*,\epsilon}:H_1\bigl(X^D_0(M\ell),\Z\bigr)_\epsilon\rightarrow H_1\bigl(X^D_0(M),\Z\bigr)_\epsilon^2\]
and an equality $\Phi_{\ell,\epsilon}={\rm coker}(\pi^*_\epsilon)/\ker(\pi_{*,\epsilon})$.

By a slight abuse of notation, from here on we will use the symbols $\pi_\ast$ and $\pi^\ast$ to denote also the analogues with $\Z_S$-coefficients of the maps of Section \ref{sec-das}. For any congruence subgroup $G$ let $S_2(G)$ denote the $\C$-vector space of weight $2$ cusp forms on $G$. Write $\T_{M\ell}^{\ell\text{-old}}$ and
$\T_{M\ell}^{\ell\text{-new}}$ for the quotients of $\T_{M\ell}$ acting faithfully, respectively, on the image $S_2^{\ell\text{-old}}\bigl(\Gamma^D_0(M\ell)\bigr)$ of the degeneracy map
\[ S_2\bigl(\Gamma^D_0(M)\bigr)\oplus S_2\bigl(\Gamma^D_0(M)\bigr)\longrightarrow S_2\bigl(\Gamma^D_0(M\ell)\bigr) \]
and on its orthogonal complement with respect to the Petersson scalar product. We keep the notations $T_q$ and $U_q$ to denote Hecke operators in $\T_M$, while $t_q$ and $u_q$ will be used for those in $\T_{M\ell}$.

Let $\m_f'$ denote the ideal of $\T_{M\ell}$ generated by $t_q-a_q$ for primes $q\nmid M\ell$, $u_q-a_q$ for primes $q|M$, $u_\ell-\delta$ and the prime $p$. Tensoring $\pi_*$ and $\pi^*$ with $\T_{M\ell}/\m_f'$ over $\T_{M\ell}$ we obtain maps
\[ \bar\pi^*:H_1\bigl(X^D_0(M),\Z_S\bigr)^2\big/\m_f'\longrightarrow H_1\bigl(X^D_0(M\ell),\Z_S\bigr)\big/\m_f'\]
and
\[ \bar\pi_*:H_1\bigl(X^D_0(M\ell),\Z_S\bigr)\big/\m_f'\longrightarrow H_1\bigl(X^D_0(M),\Z_S\bigr)^2\big/\m_f'. \]

\begin{lemma} \label{lemma-cong-1}
The map $\bar\pi_*$ is surjective.
\end{lemma}

\begin{proof} As in the proof of Proposition \ref{finite-ab-prop}, there is an exact sequence
\[ 0\longrightarrow\ker(\pi_*)\longrightarrow H_1\bigl(X^D_0(M\ell),\Z_S\bigr)\overset{\pi_*}\longrightarrow H_1\bigl(X^D_0(M),\Z_S\bigr)^2\longrightarrow\G^{\rm ab}_\ell\otimes\Z_S\longrightarrow0. \]
Since the image of $\pi_*$ is stable under $\T_{M\ell}$, the group $\G^{\rm ab}_\ell\otimes\Z_S$ inherits an action of $\T_{M\ell}$. Since $p$ does not divide the cardinality of $\Gamma^{\rm ab}_\ell$ and
the residual characteristic of $\mathfrak m'_f$ is $p$,
we have $\G^{\rm ab}_\ell/\m'_f=0$, and the result follows. \end{proof}

\begin{proposition} \label{cong-prop}
There is a canonical isomorphism
\[ {\rm coker}(\bar\pi_*\circ \bar\pi^*)\simeq \Phi_\ell/\m_f'. \]
\end{proposition}

\begin{proof} The module $\Phi_{\ell,S}$ is the quotient of ${\rm coker}(\pi^*)$ by $\ker(\pi_*)$, so it is isomorphic to the quotient of $H_1(X^D_0(M\ell),\Z_S)$ by the $\Z_S$-submodule generated by $\ker(\pi_*)$ and ${\rm im}(\pi^*)$. Hence there is an exact sequence
\begin{equation} \label{eq-cong*}
\langle \ker(\pi_*),{\rm im}(\pi^*)\rangle/\m'_f\longrightarrow H_1\bigl(X_0^D(M\ell),\Z_S\bigr)\big/\m'_f\longrightarrow\Phi_{\ell,S}/\m'_f\longrightarrow0
\end{equation}
Thanks to Lemma \ref{lemma-cong-1}, there is also an exact sequence
\[ \ker(\pi_*)/\m_f'\longrightarrow H_1\bigl(X^D_0(M\ell),\Z_S\bigr)\big/\m_f'\overset{\bar\pi_*}\longrightarrow H_1\bigl(X^D_0(M),\Z_S\bigr)^2\big/\m_f'\longrightarrow 0. \]
We conclude that $\bar\pi_*$ induces an isomorphism
\begin{equation} \label{pi}
\bar\pi_*:\bigl(H_1(X^D_0(M\ell),\Z_S)/\m_f'\bigr)\big/\langle\ker(\bar\pi_*),{\rm im}(\bar\pi^*)\rangle\overset{\simeq}\longrightarrow{\rm coker}(\bar\pi_*\circ\bar\pi^*).
\end{equation}
Since $\langle\ker(\bar\pi_*),{\rm im}(\bar\pi^*)\rangle$ is equal to the image of $\langle \ker(\pi_*),{\rm im}(\pi^*)\rangle/\m_f'$ in $H_1\bigl(X^D_0(M\ell),\Z_S\bigr)\big/\m_f'$ via the first map in \eqref{eq-cong*}, this shows that ${\rm coker}(\bar\pi_*\circ \bar\pi^*)$ is isomorphic to $\Phi_{\ell,S}/\m_f'$. Finally, since $p\not\in S$ the groups $\Phi_{\ell,S}/\m_f'$ and $\Phi_{\ell}/\m_f'$ are canonically identified, whence the claim. \end{proof}

Now we can prove the main result of this subsection.

\begin{theorem} \label{cong-thm}
Suppose that $\ell$ is an admissible prime such that $p|a_\ell-\delta(\ell+1)$ for a suitable $\delta\in\{+1,-1\}$. There exists a morphism
\[ f_\ell:\T_{M\ell}^{\text{$\ell$-new}}\longrightarrow\Z/p\Z\]
such that
\begin{itemize}
\item $f_\ell(t_q)=a_q\pmod{p}$ for all primes $q\nmid M\ell$;
\item $f_\ell(u_q)=a_q\pmod{p}$ for all primes $q|M$;
\item $f_\ell(u_\ell)=\delta\pmod{p}$.
\end{itemize}
If $\m_{f_\ell}$ denotes the kernel of $f_\ell$ then there is a group isomorphism
\begin{equation} \label{cong-thm2}
\Phi_{\ell,\epsilon}/\m_{f_\ell}\overset\simeq\longrightarrow H_1(E,\Z)_\epsilon\big/pH_1(E,\Z)_\epsilon\simeq\Z/p\Z.
\end{equation}
\end{theorem}

\begin{proof} At the cost of enlarging $S$, in this proof we assume that $\ell+1$ is invertible in $\Z_S$ (cf. Remark \ref{size-S-rem}). Then, since $\pi_*\circ\pi^*=\smallmat{\ell+1}{T_\ell}{T_\ell}{\ell+1}$, the assignment $(m,n)\mapsto (\ell+1)m-T_\ell(n)$ induces an isomorphism of groups
\begin{equation} \label{eq-cong-1}
H_1\bigl(X^D_0(M),\Z_S\bigr)^2\big/{\rm im}(\pi_*\circ\pi^*)\overset\simeq\longrightarrow H_1\bigl(X^D_0(M),\Z_S\bigr)\big/\bigl(T_\ell^2-(\ell+1)^2\bigr)
\end{equation}
which is equivariant for the action of the Hecke operators $t_q$ (respectively, $T_q$) for $q\nmid N\ell$ and $u_q$ (respectively, $U_q$) for $q|M$ on the left-hand (respectively, right-hand) side. Since $u_\ell$ acts as $\smallmat{T_\ell}{-1}{\ell}0$ on $H_1(X^D_0(M),\Z_S)^2$, we see that $x\in H_1\bigl(X^D_0(M),\Z_S\bigr)\big/pH_1\bigl(X^D_0(M),\Z_S\bigr)$ is an eigenvector for $T_\ell$ with eigenvalue $a_\ell\equiv\delta(\ell+1)\pmod p$ if and only if $(x,\delta\ell x)$ is an eigenvector for $u_\ell$ with eigenvalue $\delta$. Thanks to this and \eqref{eq-cong-1}, we find an isomorphism of groups
\begin{equation} \label{eq-cong-2}
{\rm coker}(\pi_*\circ\pi^*)/\m'_f\overset\simeq\longrightarrow H_1\bigl(X^D_0(M),\Z_S\bigr)\big/\m_f.
\end{equation}
Since ${\rm coker}(\pi_*\circ\pi^*)/\m'_f$ and ${\rm coker}(\bar\pi_*\circ\bar\pi^*)$ are canonically isomorphic, Proposition \ref{cong-prop} yields an isomorphism of groups
\begin{equation} \label{eq-cong-4}
\Phi_\ell/\m_f'\overset\simeq\longrightarrow H_1\bigl(X^D_0(M),\Z_S\bigr)\big/\m_f.
\end{equation}
It is now immediate to check that there is a canonical isomorphism
\[ H_1\bigl(X^D_0(M),\Z_S\bigr)\big/\m_f\simeq H_1\bigl(X^D_0(M),\Z_S\bigr)_f\big/pH_1\bigl(X^D_0(M),\Z_S\bigr)_f. \]
By \eqref{map1-eq}, the group $H_1\bigl(X^D_0(M),\Z_S\bigr)_f$ is isomorphic to $H_1(E,\Z_S)$. Since $p\not\in S$, isomorphism \eqref{eq-cong-4} induces an isomorphism of groups
\[ \Phi_\ell/\m'_f\overset\simeq\longrightarrow H_1(E,\Z)/pH_1(E,\Z)\simeq(\Z/p\Z)^2. \]
All the maps involved are equivariant for the action of $\tau$, so we get yet another isomorphism
\[ \Phi_{\ell,\epsilon}/\m_f'\overset\simeq\longrightarrow H_1(E,\Z)_\epsilon\big/pH_1(E,\Z)_\epsilon\simeq\Z/p\Z. \]
The action of $\T_{M\ell}$ on $\Phi_{\ell}$ is through its $\ell$-new quotient, so $\m'_f$ is fact belongs to $\T_{M\ell}^{\ell\text{-new}}$. Since $\Phi_\ell/\m_f'$ is a one-dimensional $\F_p$-vector space, the action of $\T_{M\ell}^{\ell\text{-new}}$ is given by a character $f_\ell:\T_{M\ell}^{\ell\text{-new}}\rightarrow\Z/p\Z$ whose kernel is $\m_{f_\ell}$, as was to be proved. \end{proof}

\subsection{Galois representations} \label{sec-galois-rep}

In this subsection we show the existence of an isomorphism of $G_\Q$-modules $J_\epsilon^{(\ell)}[p]/\m_{f_\ell}\simeq E[p]$ and of an isomorphism of groups $\Phi_{\ell,\epsilon}/\m_{f_\ell}\simeq H^1_{\rm sing}(K_\ell,E[p])$. Our arguments are inspired by those in \cite[\S 5.6]{BD1}. We fix an admissible prime $\ell$ and we suppose that
$p\mid a_\ell-\delta(\ell+1)$.

Write $G_{K_\ell}:=\Gal(\bar\Q_\ell/K_\ell)$ for the absolute Galois group of the local field $K_\ell$. Since we are assuming Conjecture \ref{rigid-analytic-conjecture}, there is a short exact sequence of left $\T_{M\ell}[G_{K_\ell}]$-modules
\[ 0\longrightarrow L_\epsilon\longrightarrow T_\epsilon(\bar\Q_\ell)\longrightarrow J^{(\ell)}_\epsilon(\bar\Q_\ell)\longrightarrow 0. \]
Since $L$ is a free abelian group and $T_\epsilon(\bar\Q_\ell)$ is divisible, the snake lemma implies that there is a short exact sequence of $\T_{M\ell}[G_{K_\ell}]$-modules
\begin{equation} \label{ex-seq-0}
0\longrightarrow T_\epsilon[p]\longrightarrow J_\epsilon^{(\ell)}[p]\longrightarrow L_\epsilon/p\longrightarrow 0
\end{equation}
where $T_\epsilon[p]$ and $J_\epsilon^{(\ell)}[p]$ are the $p$-torsion subgroups of $T_\epsilon(\bar\Q_\ell)$ and $J^{(\ell)}_\epsilon(\bar\Q_\ell)$, respectively. By tensoring the above exact sequence with $\T_{M\ell}/\m_{f_\ell}$ over $\T_{M\ell}$, and recalling that $p\in\m_{f_\ell}$, we find an exact sequence of $\T_{M\ell}/\fr m_{f_\ell}[G_{K_\ell}]$-modules
\[ 0\longrightarrow\bigl(T_\epsilon[p]/\m_{f_\ell}\bigr)\big/M\longrightarrow J_\epsilon^{(\ell)}[p]/\m_{f_\ell}\longrightarrow L_\epsilon/\fr m_{f_\ell}\longrightarrow0 \]
for a certain $\T_{M\ell}/\fr m_{f_\ell}[G_{K_\ell}]$-submodule $M$ of $T_\epsilon[p]/\m_{f_\ell}$. Taking $G_{K_\ell}$-cohomology of the above exact sequence yields an exact sequence of $\T_{M\ell}/\m_{f_\ell}$-modules
\begin{equation} \label{ex-seq-1}
L_\epsilon/\m_{f_\ell}\longrightarrow H^1\bigl(K_\ell,(T_\epsilon[p]/\m_{f_\ell})/M\bigr)\longrightarrow H^1\bigl(K_\ell,J_\epsilon^{(\ell)}[p]/\m_{f_\ell}\bigr)\longrightarrow H^1(K_\ell,L_\epsilon/\m_{f_\ell}).
\end{equation}
We first study the last term in \eqref{ex-seq-1}. Let $\Q_\ell^{\rm ab}$ be the maximal abelian extension of $\Q_\ell$; since $L_\epsilon/\m_{f_\ell}$ is abelian and defined over $K_\ell$, the cohomology group $H^1(K_\ell,L_\epsilon/\m_{f_\ell})$ is equal to the group of continuous homomorphisms $\Hom_{\text{cont}}\bigl(\Gal(\Q_\ell^{\rm ab}/K_\ell),L_\epsilon/\m_{f_\ell}\bigr)$. By local class field theory, there is an isomorphism
\[ \Gal(\Q_\ell^{\rm ab}/K_\ell)\simeq\hat\Z\times \cO_{K_\ell}^\times, \]
where $\cO_{K_\ell}^\times$ is the group of units in the ring of integers $\cO_{K_\ell}$ of $K_\ell$ and $\hat\Z\simeq\Gal(\Q_\ell^{\rm unr}/K_\ell)$ is (isomorphic to) the Galois group of the maximal unramified extension $K_\ell^{\rm unr}$ of $K_\ell$, which is equal to $\Q_\ell^{\rm unr}$ because the extension $K_\ell/\Q_\ell$ is unramified. Now recall the short exact sequence
\[ 0\longrightarrow \cO_{K_\ell,1}^\times\longrightarrow\cO_{K_\ell}^\times\longrightarrow(\cO_{K_\ell}/\ell\cO_{K_\ell})^\times\longrightarrow 0 \]
where $\cO_{K_\ell,1}^\times$ is the group of the elements of $\cO_{K_\ell}^\times$ which are congruent to $1$ modulo $\ell$. Since $\cl O_{K_\ell,1}^\times$ is a pro-$\ell$-group and $L_\epsilon/\m_{f_\ell}$ is $p$-torsion, the group $\Hom_{\rm cont}(\cO_{K_\ell,1}^\times,L_\epsilon/\m_{f_\ell})$ is trivial, hence
\[ \Hom_{\rm cont}(\cO_{K_\ell}^\times,L_\epsilon/\m_{f_\ell})=\Hom_{\rm cont}\bigl((\cO_{K_\ell}/\ell\cO_{K_\ell})^\times,L_\epsilon/\m_{f_\ell}\bigr)=0, \]
the second equality being due to the fact that $p\nmid\ell^2-1=|(\cO_{K_\ell}/\ell\cO_{K_\ell})^\times|$. It follows that there are canonical isomorphisms of groups
\[ \begin{split}\Hom_{\rm cont}\bigl(\Gal(\Q_\ell^{\rm ab}/K_\ell),L_\epsilon/\m_{f_\ell}\bigr)&\simeq\Hom_{\rm cont}\bigl(\Gal(\Q_\ell^{\rm unr}/K_\ell),L_\epsilon/\m_{f_\ell}\bigr)\\
   &\simeq\Hom\bigl(\Z/p\Z,L_\epsilon/\m_{f_\ell}\bigr).
   \end{split} \]
Let $\boldsymbol\mu_p$ be the group of $p$-th roots of unity in $\bar\Q_\ell$. To study the term $H^1\bigl(K_\ell,(T_\epsilon[p]/\m_{f_\ell})/M\bigr)$ in sequence
\eqref{ex-seq-1}, first recall that $T_\epsilon$ is isomorphic to $\mathbb G_m\otimes H_\epsilon$, so $T_\epsilon[p]$ is isomorphic to $\boldsymbol\mu_p\otimes H_\epsilon$ as a $G_{K_\ell}$-module. Since the structure of $\T_{M\ell}$-module on $T_\epsilon$ is given by the Hecke action on $H_\epsilon$, there is an isomorphism
\[ T_\epsilon[p]/\m_{f_\ell}\simeq\boldsymbol\mu_p\otimes(H_\epsilon/\m_{f_\ell}). \]
Furthermore, it can be easily seen that there exists a submodule $N$ of $H_\epsilon/\m_{f_\ell}$ such that the $\T_{M\ell}/\m_{f_\ell}$-module $(T_\epsilon[p]/\m_{f_\ell})/M$ is isomorphic to $\boldsymbol\mu_p\otimes\big((H_\epsilon/\m_{f_\ell})/N\big)$. Now, the group $G_{K_\ell}$ acts trivially on $H_\epsilon$ and, as a consequence of Hilbert's Theorem 90, the group $H^1(K_\ell,\boldsymbol\mu_p)$ is isomorphic to $K_\ell^\times/(K_\ell^\times)^p$. Since $p\nmid\ell^2-1$, the quotient $K_\ell^\times/(K_\ell^\times)^p$ is isomorphic to $\Z/p\Z$. We conclude that there are group isomorphisms
\[ H^1\bigl(K_\ell,(T_\epsilon[p]/\m_{f_\ell})/M\bigr)\simeq(H_\epsilon/\m_{f_\ell})/N\otimes\Z/p\Z\simeq(H_\epsilon/\m_{f_\ell})/N, \]
the second one being a consequence of the fact that $p\in\m_{f_\ell}$.

The connecting map in \eqref{ex-seq-1}, which under the above identifications can be rewritten as $L_\epsilon/\m_{f_\ell}\rightarrow(H_\epsilon/\m_{f_\ell})/N$, can be explicitly computed as follows. Let $\overline{\ker}(\pi_{\ast,\epsilon})$ be the projection of $\ker(\pi_{\ast,\epsilon})$ to $H_\epsilon$. As above, one has
\[ H^1(K_\ell,T_\epsilon[p])\simeq H^1(K_\ell,\boldsymbol\mu_p)\otimes H_\epsilon\simeq K_\ell^\times/(K_\ell^\times)^p\otimes H_\epsilon\simeq\Z/p\Z\otimes H_\epsilon\simeq H_\epsilon/p, \]
and the connecting homomorphism $L_\epsilon/p\rightarrow H^1(K_\ell,T_\epsilon[p])$ which arises by taking the $G_{K_\ell}$-cohomology of sequence \eqref{ex-seq-0} can be rewritten as $L_\epsilon/p\rightarrow H_\epsilon/p$ and is induced by composing the natural inclusion $L_\epsilon\hookrightarrow T_\epsilon(\Q_\ell)$ with the valuation map
\[ \ord_\ell:T_\epsilon(\Q_\ell)=\Q_\ell^\times\otimes H_\epsilon\xrightarrow{\ord_\ell\otimes{\rm id}}\Z\otimes H_\epsilon=H_\epsilon. \]
Thanks to \cite[Proposition 6.3]{LRV} and the fact that all the maps involved are equivariant for the action of $\tau$, we have $\ord_\ell(L_{\epsilon})=t_r\bigl(\overline{\ker}(\pi_{\ast,\epsilon})\bigr)$ where $t_r:=T_r-r-1$.

Since the Galois action commutes with the Hecke action, it follows that the image of the connecting homomorphism $L_\epsilon/\m_{f_\ell}\rightarrow(H_\epsilon/\m_{f_\ell})/N$ is $t_r\bigl(\overline{\ker}(\pi_{\ast,\epsilon})/\m_{f_\ell}\bigr)$. The endomorphism $t_r$ of $\overline{\ker}(\pi_{\ast,\epsilon})/\m_{f_\ell}$ is just multiplication by the reduction modulo $p$ of $a_r-(r+1)$, which is an isomorphism because $p\nmid a_r-(r+1)$ by Assumption \ref{ass}. Hence $t_r$
takes $\overline{\ker}(\pi_{\ast,\epsilon})/\m_{f_\ell}$ isomorphically onto its image and induces an isomorphism
\[ (H_\epsilon/\m_{f_\ell})\big/\bigl(\overline{\ker}(\pi_{\ast,\epsilon})/\m_{f_\ell}\bigr)\overset\simeq\longrightarrow(H_\epsilon/\m_{f_\ell})\big/t_r\bigl(\overline{\ker}(\pi_{\ast,\epsilon})/\m_{f_\ell}\bigr). \]
Now recall that, by definition, $\Phi_{\ell,\epsilon}:={\rm coker}(f_\epsilon^\ast)/\ker(\pi_{\ast,\epsilon})$, so $\Phi_{\ell,\epsilon}/\m_{f_\ell}$ is isomorphic to the quotient of ${\rm coker}(f_\epsilon^\ast)/\m_{f_\ell}$ by the image of $\ker(\pi_{\ast,\epsilon})/\m_{f_\ell}$. This last quotient maps surjectively onto $(H_\epsilon/\m_{f_\ell})\big/\bigl(\overline{\ker}(\pi_{\ast,\epsilon})/\m_{f_\ell}\bigr)$ and thus there exists a canonical surjective homomorphism
\[ \Phi_{\ell,\epsilon}/\m_{f_\ell}\;\longepi\;(H_\epsilon/\m_{f_\ell})\big/t_r\bigl(\overline{\ker}(\pi_{\ast,\epsilon})/\m_{f_\ell}\bigr). \]
The exact sequence of $\T_{M\ell}/\m_{f_\ell}$-modules \eqref{ex-seq-1} can therefore be rewritten as
\begin{equation} \label{ex-seq-2}
0\longrightarrow\Psi\longrightarrow H^1\bigl(K_\ell,J_\epsilon^{(\ell)}[p]/\m_{f_\ell}\bigr)\longrightarrow\Hom_\text{cont}\bigl(\Gal(\Q_\ell^{\rm unr}/K_\ell),L_\epsilon/\m_{f_\ell}\bigr)
\end{equation}
where $\Psi$ is a suitable quotient of $\Phi_{\ell,\epsilon}/\m_{f_\ell}$.

\begin{theorem} \label{teo-rep}
\begin{enumerate}
\item The $G_\Q$-modules $J_\epsilon^{(\ell)}[p]/\m_{f_\ell}$ and $E[p]$ are isomorphic.
\item The groups $\Phi_{\ell,\epsilon}/\m_{f_\ell}$ and $H^1_{\rm sing}(K_\ell,E[p])$ are isomorphic.
\item Exact sequence \eqref{ex-seq-2} can be rewritten as
\[ 0\longrightarrow\Phi_{\ell,\epsilon}/\m_{f_\ell}\longrightarrow H^1(K_\ell,E[p])\longrightarrow\Hom_{\rm cont}\bigl(\Gal(\Q_\ell^{\rm unr}/K_\ell),L_\epsilon/\m_{f_\ell}\bigr). \]
\end{enumerate}
\end{theorem}
\begin{proof} By \cite{BLR} and the Eichler--Shimura relations, the quotient $J_\epsilon^{(\ell)}[p]/\m_{f_\ell }$ is isomorphic as a $G_\Q$-module to the direct sum of $h\geq1$ copies of $E[p]$. By \cite[Lemma 2.6]{BD1}, the $\F_p$-vector space $H^1(K_\ell,E[p])$ has dimension $2$ and can be (non-canonically) decomposed into a sum
\[ H^1(K_\ell,E[p])=H^1_{\rm fin}(K_\ell,E[p])\oplus H^1_{\rm sing}(K_\ell,E[p]) \]
of one-dimensional subspaces. The image of $H^1_{\rm sing}(K_\ell,E[p])$ in the group of continuous homomorphisms in exact sequence \eqref{ex-seq-2} is trivial. Since
$\dim_{\F_p}(\Psi)\leq\dim_{\F_p}(\Phi_{\ell,\epsilon}/\m_{f_\ell})$ and $\dim_{\F_p}(\Phi_{\ell,\epsilon}/\m_{f_\ell})=1$ by the last claim of Theorem \ref{cong-thm}, we conclude that $h=1$ and
\[ \Psi\simeq\Phi_{\ell,\epsilon}/\m_{f_\ell}\simeq H^1_{\rm sing}(K_\ell,E[p]), \]
from which all the statements follow. \end{proof}
In light of Theorem \ref{teo-rep}, from here on we fix an isomorphism
\begin{equation} \label{torsion-isom-eq}
J_\epsilon^{(\ell)}[p]/\m_{f_\ell}\simeq E[p]
\end{equation}
of $G_\Q$-modules and an isomorphism
\begin{equation} \label{Phi-sing-isom-eq}
\Phi_{\ell,\epsilon}/\m_{f_\ell}\simeq H^1_{\rm sing}(K_\ell,E[p])
\end{equation}
of $\F_p$-vector spaces.

\section{Gross--Zagier type formula and Darmon points} \label{formula-sec}

In this section assume that $D>1$. Fix throughout an admissible prime $\ell$, set
\[ \Gamma:=\Gamma_\ell \]
for the Ihara group at $\ell$ and denote by $t$ the exponent $t_\ell$ of $\Gamma^{\rm ab}$. Building on the arguments and constructions of \cite{LRV}, in this section we prove our Gross--Zagier type formula (Theorem \ref{rec-law}) relating the class modulo $p$ of $\cl L_K(E,\chi,1)$ to a certain twisted sum of Darmon points. This is a generalization to the case of division quaternion algebras and arbitrary characters of the formula proved in \cite[Theorem 3.9]{BDD}. In fact, a suitable extension of the arguments with modular symbols and specializations of Stark--Heegner points described in \cite[\S 3.3]{BDD} yields the analogue of Theorem \ref{rec-law} in the $D=1$ setting.

\subsection{Auxiliary results}

Recall from \S \ref{sec-galois-rep} and the proof of Theorem \ref{teo-rep} that the cokernel of the map arising from the composition of the inclusion $L_\epsilon\subset T_\epsilon(\Q_\ell)$, the valuation map $\ord_\ell:T_\epsilon(\Q_\ell)\rightarrow H_\epsilon$ and the projection $H_\e\twoheadrightarrow H_\e/\m_{f_\ell}$, which is denoted by $\Psi$ in \eqref{ex-seq-2}, is a non-trivial $\F_p$-vector space isomorphic to $\Phi_{\ell,\epsilon}/\m_{f_\ell}$. For any unramified extension $W/K_\ell$ denote by
\begin{equation} \label{partial-eq}
\partial_\ell:J_\epsilon^{(\ell)}(W)\longrightarrow \Phi_{\ell,\epsilon}/\m_{f_\ell}
\end{equation}
the map that is obtained by composing the inverse of isomorphism \eqref{rigid-isom-eq} with the valuation map $\ord_\ell:T_\epsilon(W)/L_\epsilon\rightarrow H_\epsilon/\ord_\ell(L_\epsilon)$, the canonical projection to $\Psi$ and the isomorphism of this $\F_p$-vector space with $\Phi_{\ell,\e}/\m_{f_\ell}$.

Let
\[ r:\mathcal H_\ell\longrightarrow \mathcal T \]
denote the $\GL_2(\Q_\ell)$-equivariant reduction map (see, e.g., \cite[\S 5.1]{Dar}) and fix a base point $\tau\in K_\ell-\Q_\ell$ such that $r(\tau)=v_*$. Let $\g_1\in\Gamma$ and let $\{e_0,\dots,e_n\}$ be a set of edges $e_i\in\cl E^+$ such that
\begin{itemize}
\item $s(e_1)=v_\ast$, $s(e_n)=\gamma_1(v_\ast)=:v_n$;
\item $t(e_i)=t(e_{i+1})=:v_i$ for \emph{odd} indices in $\{1,\dots,n-1\}$;
\item $s(e_i)=s(e_{i+1})=:v_i$ for \emph{even} indices in $\{2,\dots,n-2\}$.
\end{itemize}
Notice that, in the above, the integer $n$ is always even. If $\g_2\in\Gamma$ then, by \cite[Proposition 5.2]{LRV}, there is an equality
\begin{equation} \label{int-eq}
\ord_\ell\bigg(\mint_{\PP^1(\Q_\ell)}\frac{t-\gamma_1^{-1}(\tau)}{t-\tau}d\mu_{\gamma_2}^{\mathcal{Y}}(t)\bigg)=\sum_{i=0}^n(-1)^i\mu_{\g_2}^{\mathcal{Y}}(e_i)
\end{equation}
of elements in $H$, where $\mu^{\mathcal Y}$ is the cocycle introduced in Definition \ref{mu}.

\begin{remarkwr}
In the following we adopt the identification $H_1\bigl(\Ga,\Z_S\bigr)=\Ga^{\rm ab}\otimes\Z_S$ and write $[\g]$ for the natural image in $H_1\bigl(\Ga,\Z_S\bigr)$ of an element $\g\in\Ga$.
\end{remarkwr}

Now we introduce the $1$-cocycle
\[ \tilde{m}^{\mathcal Y}\in Z^1\Big(\G,\cF\bigl(\V,H_1\bigl(\Ga,\Z_S\bigr)\bigr)\Big) \]
defined by the rule
\[ \tilde{m}_{\g}^{\mathcal Y}(v):=[g_{\g,v}] \]
where $g_{\g,v}\in\Ga$ is given by the formula
\[ \quad g_{\g,v}:=\begin{cases}\g_v\g\g^{-1}_{\g^{-1}(v)} & \text{if $v\in \V^+$}\\[2mm]
                   \om_{\ell}^{-1}\g_v\g\g^{-1}_{\g^{-1}(v)}\om_\ell & \text{if $v\in \V^-$.}
                   \end{cases} \]
Note that $\g_v\g\g^{-1}_{\g^{-1}(v)}$ stabilizes $v_\ast$ (respectively, $\hat{v}_\ast$), and thus lies in $\Ga$ (respectively, in $\hGa$), if $v\in\V^+$ (respectively, $v\in\V^-$). Hence $g_{\g,v}$ always lies in $\Ga$. We leave it to the reader to check that $\tilde{m}^{\mathcal Y}$ is a well-defined cocycle;
see Definition \ref{mu} and \cite[\S 4]{LRV} for a similar construction.

Consider the composition
\[ {\rm pr}_1: H_1\bigl(X_0^D(M),\Z_S\bigr)^2\twoheadrightarrow H_1\bigl(X_0^D(M),\Z_S\bigr)^2\big/\m_f'\twoheadrightarrow{\rm coker}(\bar\pi_*\circ\bar\pi^*)\twoheadrightarrow\Phi_{\ell,\epsilon}/\mathfrak m_{f_\ell}\simeq\Z/p\Z \]
where the first two maps are the canonical projections, the third is induced by Proposition \ref{cong-prop} and the isomorphism is that of \eqref{cong-thm2}. If $e\in\E$ then set
\begin{equation} \label{tildemu}
\tilde{\mu}^{\mathcal Y}_{\g}(e):={\rm pr}_1\bigl(\tilde{m}_\g^{\mathcal Y}(s(e)),\tilde{m}_\g^{\mathcal Y}(t(e))\bigr).
\end{equation}
Similarly, define also the composition
\[ {\rm pr}_2:H_1\bigl(X_0^D(M),\Z_S\bigr)\twoheadrightarrow H_1\bigl(X_0^D(M),\Z_S\bigr)\big/\m_f\simeq{\rm coker}(\bar\pi_*\circ\bar\pi^*)\twoheadrightarrow\Phi_{\ell,\epsilon}/\mathfrak m_{f_\ell}\simeq\Z/p\Z \]
where the first isomorphism is \eqref{eq-cong-2}. Recall from condition $5$ in Assumption \ref{ass} that there exists $\delta\in\{\pm1\}$ such that $p|a_{\ell}+\delta (\ell +1)$. The isomorphism in \eqref{eq-cong-1} is induced by the map $(x,y)\mapsto(\ell+1)x-T_\ell(y)$; since $p|a_\ell-\delta(\ell+1)$, this map is just $(x,y)\mapsto(\ell+1)(x-\delta y)$ from $H_1\bigl(X_0^D(M),\Z_S\bigr)^2\big/\m_f'$ to $H_1\bigl(X_0^D(M),\Z_S\bigr)\big/\m_f$. It follows that
\begin{equation} \label{note}
\tilde{\mu}^{\mathcal Y}_\g(e)=(\ell+1){\rm pr}_2\bigl(\tilde{m}_\g^{\mathcal Y}(t(e))-\delta\tilde{m}_\g^{\mathcal Y}(s(e))\bigr).
\end{equation}
We thus obtain that $\tilde{\mu}^{\mathcal Y}$ is also well defined with values in $\cF_0(\E,\Z/p\Z)$.

Finally, introduce the map
\[ {\rm pr}_3:H_1\bigl(X^D_0(M\ell),\Z_S\bigr)\;\longepi\;{\rm coker}(\bar\pi_*\circ\bar\pi^*)\;\longepi\;\Phi_{\ell,\epsilon}/\mathfrak m_{f_\ell}\simeq\Z/p\Z \]
where the first arrow is the composition of the canonical projection
\[ H_1\bigl(X^D_0(M\ell),\Z_S\bigr)\;\longepi\;\bigl(H_1(X^D_0(M\ell),\Z_S)/\m_f'\bigr)\big/\langle\ker(\bar\pi_*),{\rm im}(\bar\pi^*)\rangle \]
with isomorphism \eqref{pi}, and define
\[ \bar\mu^{\mathcal Y}:={\rm pr}_3(\mu^{\mathcal Y}). \]

\begin{lemma} \label{Lemma7.1}
$\bar\mu^{\mathcal Y}=\tilde\mu^{\mathcal Y}$.
\end{lemma}
\begin{proof} Fix $\g\in\G$ and $e\in\E^+$ and let $g_{\g,e}\in\Gap$ be such that $\g_e\g=g_{\g,e}\g_{e'}$ for some $e'\in\E^+$. By Definition \ref{mu}, one has
\[ \bar\mu^{\mathcal Y}_\g(e)={\rm pr}_3\bigl([g_{\g,e}]\bigr), \]
while by \eqref{tildemu} there is an equality
\[ \tilde\mu^\mathcal{Y}_\g(e)={\rm pr}_1\bigl([g_{\g,e}],[\om_\ell^{-1}g_{\g,e}\om_\ell]\bigr). \]
By construction, there is a commutative triangle
\[ \xymatrix@C=25pt@R=20pt{H_1\bigl(X^D_0(M\ell),\Z_S\bigr)\ar[d]\ar@{->>}[r]&\mathrm{coker}(\bar\pi_*\circ\bar\pi^*)\\
                           H_1\bigl(X_0^D(M),\Z_S\bigr)^2\ar@{->>}[ur]&} \]
where the vertical arrow is induced by the map $\Gap\rightarrow\Ga^2$ taking $\g$ to $(\g,\om_\ell^{-1}\g \om_\ell)$ via the canonical projections and the other two maps are the surjections already appearing in the definitions of $\mathrm{pr}_1$ and $\mathrm{pr}_3$. This shows the required equality for even edges, and the analogous equality for odd edges follows similarly. \end{proof}

Let us denote by $\partial'_\ell$ the composition of the map $\partial_\ell$ in \eqref{partial-eq} with the isomorphism \eqref{cong-thm2} between $\Phi_{\ell,\e}/\m_{f_\ell}$ and $\Z/p\Z$. Fix now $\psi\in\Emb(\cO,R)$ and choose $\tau:=z_\psi$, where, as in \S \ref{SH-subsection}, $z_\psi\in\cl H_\ell\cap K_\ell$ is the (unique) point such that $\psi(\alpha)\binom{z_\psi}{1}=\alpha\binom{z_\psi}{1}$ for all $\alpha\in K$. Let us also write $d_\epsilon$ for the composition of the $2$-cocycle $d=d_\tau$ introduced in \eqref{desc-d} with the map $T(K_\ell)\rightarrow J_\epsilon^{(\ell)}(K_\ell)$ defined in the obvious way. Similarly, if $\beta$ is as in \eqref{def-beta} then let $\beta_\epsilon:\Gamma\rightarrow J_\epsilon^{(\ell)}(K_\ell)$ be the induced map. Observe that, with this notation in force, Definition \ref{SH-J-dfn} reads
\begin{equation} \label{darmon-dfn}
P_\psi^\epsilon:=t\cdot\beta_\epsilon(\gamma_\psi)\in J_\epsilon^{(\ell)}(K_\ell).
\end{equation}
It is worthwhile to explicitly remark that in this section we view the Darmon points $P_\psi^\epsilon$ as rational over the local field $K_\ell$. In fact, the Gross--Zagier type results we are about to prove are of a genuinely local nature, so we do not need to assume that the points we work with are global, as predicted by Conjecture \ref{SH-conjecture}.

From \eqref{int-eq} and Lemma \ref{Lemma7.1} we obtain equalities
\begin{equation} \label{partial'}
\partial_\ell'\bigl(d_\epsilon(\gamma_1,\gamma_2)\bigr)=\sum_{i=0}^n(-1)^i\bar\mu_{\g_2}^{\mathcal Y}(e_i)=\sum_{i=0}^n(-1)^i\tilde\mu_{\g_2}^{\mathcal Y}(e_i),
\end{equation}
with the edges $e_i$ being defined as for equality \eqref{int-eq}; namely, the $e_i\in\cl E^+$ satisfy
\begin{itemize}
\item $s(e_1)=v_\ast$, $s(e_n)=\gamma_1^{-1}(v_\ast)=:v_n$;
\item $t(e_i)=t(e_{i+1})=:v_i$ for \emph{odd} indices in $\{1,\dots,n-1\}$;
\item $s(e_i)=s(e_{i+1})=:v_i$ for \emph{even} indices in $\{2,\dots,n-2\}$.
\end{itemize}
Define a function $\alpha_\tau:\Gamma\rightarrow \Z/p\Z$ by setting
\[ \alpha_\tau(\g):=-(\ell+1){\rm pr}_2\bigl(\tilde m_\g^{\mathcal Y}(v_*)\bigr). \]
Observe that, by definition, $\alpha_\tau=\alpha_{\tau'}$ for all $\tau'$ with $r(\tau')=v_*$. Recall the element $\gamma_\psi\in\Gamma_0^D(M)$ attached to the embedding $\psi$ as in \S \ref{SH-subsection}.

\begin{lemma} \label{lemma-partial}
Suppose $\delta=-1$. The equality
\[ \partial'_\ell(P_\psi^\epsilon)=t\cdot\alpha_\tau(\g_\psi) \]
holds in $\Z/p\Z$.
\end{lemma}
\begin{proof} Fix $\g_1,\g_2\in\G$ and $e\in\E$. Choose a sequence $\{e_0,\dots,e_n\}$ of even edges joining the vertices $v_*$ and $\g_1^{-1}(v_*)$ as in \eqref{partial'}. Since $\delta=-1$, by \eqref{note} there is an equality
\[ \sum_{i=0}^n(-1)^i\tilde\mu^{\cl Y}_{\g_2}(e_i)=(\ell+1)\sum_{i=0}^n(-1)^i{\rm pr}_2\bigl(\tilde{m}_\g^{\mathcal Y}(t(e))+\tilde{m}_\g^{\mathcal Y}(s(e))\bigr). \]
The terms in the right-hand sum cancel out telescopically and we find that
\begin{equation} \label{eq0}
\sum_{i=0}^n(-1)^i\tilde\mu^{\cl Y}_{\g_2}(e_i)=-(\ell+1){\rm pr}_2\bigl(\tilde m^{\cl Y}_\g(t(e_n))-\tilde m_\g^{\cl Y}(s(e_0))\bigr).
\end{equation}
Observe that
\begin{equation} \label{eq1}
\begin{split}
\tilde m_{\g_1\g_2}^{\cl Y}(v_*)-\tilde m_{\g_1}^{\cl Y}(v_*)&=\bigl[\g_1\g_2\g^{-1}_{\g_2^{-1}\g_1^{-1}(v_*)}\bigr]-\bigl[\g_1\g^{-1}_{\g_1^{-1}(v_*)}\bigr]\\
&=\bigl[\g_{\g_1^{-1}(v_*)}\g_2\g^{-1}_{\g_2^{-1}\g_1^{-1}(v_*)}\bigr]=\tilde m_{\g_2}^{\cl Y}\bigl(\g_1^{-1}(v_*)\bigr).
\end{split}
\end{equation}
Combining \eqref{partial'}, \eqref{eq0} and \eqref{eq1} we obtain
\begin{equation} \label{eq-partial'}
\partial_\ell'\bigl(d_\epsilon(\gamma_1,\gamma_2)\bigr)=\alpha_{\tau}(\g_1\g_2)-\alpha_{\tau}(\g_1)-\alpha_{\tau}(\g_2).
\end{equation}
It is then a consequence of equations \eqref{def-beta} and \eqref{eq-partial'} that both $\partial'_\ell\circ\beta_\epsilon$ and $\alpha_\tau$ split the $2$-cocycle $\partial'_\ell\circ d_\epsilon\in Z^2(\Gamma,\Z/p\Z)$, whence
\begin{equation} \label{darmon-eq}
\partial'_\ell\bigl(t\cdot\beta_\epsilon(\gamma)\bigr)=t\cdot\alpha_\tau(\gamma)
\end{equation}
for all $\gamma\in\Gamma$ because $\partial'_\ell$ is a group homomorphism. In light of \eqref{darmon-dfn}, the claim of the lemma follows upon taking $\gamma=\gamma_\psi$ in equality \eqref{darmon-eq}.\end{proof}

\subsection{A Gross--Zagier formula} \label{reciprocity-subsec}

Recall the set $\{\psi_\sigma\mid\sigma\in G_c\}$ of representatives for the $\Gamma_0^D(M)$-equivalence classes of optimal embeddings of $\cO_c$ into $R$ fixed in \S \ref{oriented-subsec}. For simplicity, set $\tau_\sigma:=z_{\psi_\sigma}$ and $v_\sigma:=r(\tau_\sigma)$ for all $\sigma\in G_c$. Since the reduction map is $\Gamma$-equivariant and $\ell$ is prime to $c$, the stabilizer of $v_\sigma$ in $\GL_2(\Q_\ell)$ coincides with $\GL_2(\Z_\ell)$, hence $v_\sigma=v_*$ for all $\sigma\in G_c$. Define
\begin{equation} \label{tau-chi}
P_\chi^\epsilon:=\sum_{\sigma\in G_c}P_{\psi_\sigma}^\epsilon\otimes\chi^{-1}(\sigma)\in J_\epsilon^{(\ell)}(K_\ell)\otimes\Z[\chi]_S
\end{equation}
and, again to ease the writing, set $\g_\sigma:=\g_{\psi_\sigma}\in\Gamma_0^D(M)$ for all $\sigma\in G_c$.

Let $[\star]$ be the class of the element $\star$ in a quotient group. Now we can prove our Gross--Zagier type formula for the (algebraic part of the) special value $L_K(E,\chi,1)$, which can also be regarded as an explicit reciprocity law in the sense of \cite{BD1}.

\begin{theorem}\label{rec-law}
Suppose $\delta=-1$. Then
\[ (\partial'_\ell\otimes{\rm id})(P_\chi^\epsilon)=t\cdot\bigl[\cl L_K(E,\chi,1)\bigr] \]
in $\Z[\chi]_S/p\Z[\chi]_S$.
\end{theorem}

\begin{proof} Combining Lemma \ref{lemma-partial} with the fact that $v_\sigma=v_*$ for all $\sigma\in G_c$ gives
\begin{equation} \label{par-eq}
(\partial'_\ell\otimes{\rm id})(P_\chi^\epsilon)=t\cdot\sum_{\sigma\in G_c}\alpha_{\tau_\sigma}(\g_\sigma)\otimes\chi^{-1}(\sigma)
\end{equation}
in $\Z[\chi]_S/p\Z[\chi]_S$. Since $\alpha_{\tau_\sigma}(\g_\sigma)={\rm pr}_2\bigl([\g_\sigma]\bigr)$, by definition of $\cl L_K(E,\chi,1)$ one has
\[ \sum_{\sigma\in G_c}\alpha_{\tau_\sigma}(\g_\sigma)\otimes\chi^{-1}(\sigma)=\bigl[\cl L_K(E,\chi,1)\bigr] \]
in $\Z[\chi]_S/p\Z[\chi]_S$. The result then follows from equality \eqref{par-eq}. \end{proof}

\section{Arithmetic results and consequences} \label{proofs-sec}

With our special value formula (Theorem \ref{rec-law}) at hand, in this section we prove the results on the vanishing of the Selmer groups and on the Birch and Swinnerton-Dyer conjecture for $E$ in the case of analytic rank $0$ that were anticipated in the introduction.

From here on we shall assume the validity of Conjecture \ref{SH-conjecture}.

\subsection{A result on local Kummer maps}

Quite generally, let $F$ be a number field and let
\[ \kappa:J_\epsilon^{(\ell)}(F)\longrightarrow H^1\bigl(F,J_\epsilon^{(\ell)}[p]\bigr) \]
be the Kummer map relative to $J_\epsilon^{(\ell)}$. Composing $\kappa$ with the maps induced by the canonical projection $J_\epsilon^{(\ell)}[p]\rightarrow J_\epsilon^{(\ell)}[p]/\m_{f_\ell}$ and by isomorphism \eqref{torsion-isom-eq} yields a map
\begin{equation} \label{bar-kappa-eq}
\bar\kappa:J_\epsilon^{(\ell)}(F)\longrightarrow H^1(F,E[p]).
\end{equation}
By a slight abuse of notation, we adopt the symbol $\bar\kappa$ also for the map
\[ \bar\kappa:J_\epsilon^{(\ell)}(K_\ell)\longrightarrow H^1(K_\ell,E[p]) \]
which is obtained by considering the local counterpart of the Kummer map $\kappa$ and viewing \eqref{torsion-isom-eq} as an isomorphism of $\Gal(\bar\Q_\ell/K_\ell)$-modules via the inclusion $\Gal(\bar\Q_\ell/K_\ell)\hookrightarrow G_\Q$ induced by the injection $\bar\Q\hookrightarrow\bar\Q_\ell$ fixed at the outset.

If $q$ is a prime number let $\text{res}_q:H^1(F,E[p])\rightarrow H^1(F_q,E[p])$ be the restriction map and let
\[ \delta_q:E(F_q)\longrightarrow H^1(F_q,E[p]),\qquad\kappa_q:J_\epsilon^{(\ell)}(F_q)\longrightarrow H^1\bigl(F_q,J_\epsilon^{(\ell)}[p]\bigr) \]
be the local Kummer maps at $q$ relative to $E$ and $J^{(\ell)}_\epsilon$, respectively. Finally, for any prime $\p$ of $F$ above $p$ let $\nu_\p$ be the (normalized) valuation of $F_\p$ and let $e_\p:=\nu_\p(p)$ be the absolute ramification index of $F_\p$ (in particular, $e_\p=1$ if $p$ is unramified in $F$).

\begin{proposition}
Assume that $e_\p<p-1$ for all $\p|p$. If $P\in J^{(\ell)}_\epsilon(F)$ then
\[ \mathrm{res}_q\bigl(\bar\kappa(P)\bigr)\in\mathrm{Im}(\delta_q) \]
for all primes $q\nmid M\ell$.
\end{proposition}

A proof of this proposition, obtained by combining the description of the image of the local Kummer maps above $p$ in terms of flat cohomology given in \cite[\S 3.3]{LV} with classical results of Raynaud on $p$-torsion group schemes (\cite{Ray}), can be found in \cite[Proposition 5.2]{LV}.

\subsection{Linear algebra preliminaries} \label{preliminaries-subsec}

The goal of this subsection is to recall the arguments in \cite[\S 8]{LV} and introduce the technical tools (Propositions \ref{alg-4} and \ref{lemma-II}) that will be needed to prove the main arithmetic theorems of this paper.

Let $\chi\in\widehat{G}_c$ be our complex-valued character of $G_c$. Since $p\not\in S$ by condition $1$ in Assumption \ref{ass}, every prime ideal $\fr p$ of $\Z[\chi]$ above $p$ determines a prime ideal $\fr p_S:=\fr p\Z[\chi]_S$ of $\Z[\chi]_S$.
\begin{lemma} \label{completions-lemma}
Let $\fr p$ be a prime ideal of $\Z[\chi]$ above $p$. The completion of $\Z[\chi]$ at $\fr p$ is canonically isomorphic to the completion of $\Z[\chi]_S$ at $\fr p_S$.
\end{lemma}
\begin{proof} For all integers $n\geq1$ write $\bar S_n$ for the multiplicative system of $\Z[\chi]/\fr p^n$ which is the image of $S$ under the natural projection. For every $n\geq1$ there is a canonical ring isomorphism
\begin{equation} \label{localizations-iso-eq}
\bigl(\Z[\chi]/\fr p^n\bigr)_{\bar S_n}\simeq\Z[\chi]_S/\fr p_S^n.
\end{equation}
But the elements of $\bar S_n$ are invertible in $\Z[\chi]/\fr p^n$ since $p$ does not belong to $S$, hence the localization $\bigl(\Z[\chi]/\fr p^n\bigr)_{\bar S_n}$ canonically identifies with $\Z[\chi]/\fr p^n$. In light of \eqref{localizations-iso-eq}, the lemma is proved by passing to the inverse limit. \end{proof}
Choose a prime ideal $\fr p$ of $\Z[\chi]$ above $p$ such that
\begin{equation} \label{alg-nonzero-eq}
\text{the image of $\cl L_K(E,\chi,1)$ in $\Z[\chi]_S/\fr p_S$ is not zero.}
\end{equation}
This can be done thanks to condition $4$ in Assumption \ref{ass}. Denote by $\W$ the $\fr p$-adic completion of $\Z[\chi]$. The prime $p$ is unramified in $\Z[\chi]$ since it does not divide $h^+(c)$ by condition $2$ in Assumption \ref{ass}, hence the ideal $p\W$ is the maximal ideal of $\W$; in particular, we conclude from Lemma
\ref{completions-lemma} that
\[ \Z[\chi]_S/\fr p_S=\W/p\W. \]
For any $\Z[G_c]$-module $M$ write $M\otimes_\chi\C$ (respectively, $M\otimes_\chi\W$) for the tensor product of the $\Z[G_c]$-modules $M$ and $\C$ (respectively, $M$ and $\W$), where the structure of $\Z[G_c]$-module on $\C$ (respectively, $\W$) is induced by $\chi$. As in the introduction, if $M$ is a $\Z[G_c]$-module define also
\[ M^\chi:=\bigl\{x\in M\otimes_\Z\mathcal\C\mid\text{$\sigma(x)=\chi(\sigma)x$ for all $\sigma\in G_c$}\bigr\}, \]
so that there is a canonical identification
\[ M^\chi=M\otimes_\chi\C \]
of $\C[G_c]$-modules (for a proof of this fact see, e.g., \cite[Proposition 8.1]{LV}).

Choose once and for all an (algebraic) isomorphism $\C_p\simeq\C$ which is the identity on $\Z[\chi]$. Henceforth we shall view $\C$ as a $\W$-module via this isomorphism, obtaining an isomorphism
\[ \bigl(E(H_c)\otimes_\chi\W\bigr)\otimes_\W\C\simeq E(H_c)\otimes_\chi\C. \]
The following flatness result will be frequently used in the sequel.

\begin{lemma} \label{alg-2}
The module $\W$ is flat over $\Z[G_c]$, and every $\F_p[G_c]$-module is flat.
\end{lemma}

\begin{proof} First of all, $\W$ is flat over $\Z$. Moreover, if $\ell$ is a prime number dividing $h^+(c)$ then $\ell\not=p$, hence $\W/\ell\W=0$. The flatness of $\W$ follows from \cite[Theorem 1.6]{BG}. The second assertion can be shown in the same way. \end{proof}

The next statement is proved exactly as \cite[Proposition 8.3]{LV}.

\begin{proposition} \label{alg-4}
If $\Sel_p(E/H_c)\otimes_\chi\W=0$ then $E(H_c)^\chi=0$.
\end{proposition}

Thus the triviality of $E(H_c)^\chi$ is guaranteed by that of $\Sel_p(E/H_c)\otimes_\chi\W$.

The rest of this subsection is devoted to a couple of further algebraic lemmas which are needed to prove the vanishing of the twisted $p$-Selmer groups; this part follows \cite[\S 8.2]{LV} closely, so we will merely sketch the arguments and refer to \emph{loc. cit.} for complete proofs.

In the following, use the symbol $\chi$ also to denote the $\Z$-linear extension
\[ \Z[G_c]\overset{\chi}{\longrightarrow}\Z[\chi]\subset\W \]
of the character $\chi$. Composing $\chi$ with the projection onto $\W/p\W$ yields a homomorphism which factors through $\F_p[G_c]=\Z[G_c]/p\Z[G_c]$, and we define
$\chi_p:\F_p[G_c]\rightarrow\W/p\W$ to be the resulting map. In particular, the homomorphism $\chi_p$ gives $\W/p\W$ a structure of $\F_p[G_c]$-module (which is nothing other than the structure induced naturally by that of $\Z[G_c]$-module on $\W$), and for an $\F_p[G_c]$-module $M$ the notation $M\otimes_{\chi_p}(\W/p\W)$ will indicate that the tensor product is taken over $\F_p[G_c]$ with respect to $\chi_p$.

Set $I_{\chi_p}:=\ker(\chi_p)$ and for any $\F_p[G_c]$-module $M$ let $M[I_{\chi_p}]$ be the $I_{\chi_p}$-torsion submodule of $M$, i.e. the submodule of $M$ which is annihilated by all the elements of $I_{\chi_p}$. Finally, adopt similar notations and conventions for the map $\chi_p^{-1}:\F_p[G_c]\rightarrow\W/p\W$ which is
induced by the inverse character to $\chi$.

The flatness result of Lemma \ref{alg-2} yields the following important facts:
\begin{itemize}
\item for every $\F_p[G_c]$-module $M$ there are canonical identifications
\[ M\otimes_\chi\W=M\otimes_{\chi_p}(\W/p\W)=M[I_{\chi_p}]\otimes_{\chi_p}(\W/p\W)=M[I_{\chi_p}]\otimes_\chi\W \]
of $\W$-modules (\cite[Lemma 8.4]{LV});
\item if $M$ is an $\F_p[G_c]$-module then $M[I_{\chi_p}]$ injects into $M\otimes_\chi\W$ amd $M\bigl[I_{\chi^{-1}_p}\bigr]$ injects into $M\bigl[I_{\chi^{-1}_p}\bigr]\otimes_\chi\W$ (\cite[Lemma 8.5]{LV}).
\end{itemize}
As a consequence, the linear algebra results in \cite[\S 8.2]{LV} carry over \emph{verbatim} to our real quadratic setting; here we content ourselves with recalling the proof of a crucial statement about the non-triviality of (the dual of) a certain restriction map in Galois cohomology.

To begin with, for any $\F_p$-vector space $V$ denote the $\F_p$-dual of $V$ by
\[ V^\vee:=\Hom_{\F_p}(V,\F_p). \]
The dual of an $\F_p[G_c]$-module inherits a natural structure of $\F_p[G_c]$-module: a Galois element $\sigma$ acts on a homomorphism $\varphi$ by $\sigma(\varphi):=\varphi\circ\sigma^{-1}$. Furthermore, if $f$ is a map of $\F_p[G_c]$-modules then its dual $f^\vee$ is again $G_c$-equivariant. It can be immediately checked that if an $\F_p[G_c]$-module is of $I_{\chi_p}$-torsion then its dual is of $I_{\chi_p^{-1}}$-torsion.

Let $\ell$ be an admissible prime and let
\[ \res_\ell:\Sel_p(E/H_c)\longrightarrow H^1_{\rm fin}(H_{c,\ell},E[p]) \]
be the natural restriction map; with a slight abuse of notation, we will adopt the same symbol also for the map
\[ \res_\ell:\Sel_p(E/H_c)[I_{\chi_p}]\longrightarrow H^1_{\rm fin}(H_{c,\ell},E[p])[I_{\chi_p}] \]
between the $I_{\chi_p}$-torsion submodules which is induced by the previous one.

\begin{lemma} \label{aux-lemma-I}
If there exists $s\in\Sel_p(E/H_c)[I_{\chi_p}]$ such that $\res_\ell(s)\neq0$ then the map
\[ \res_\ell^\vee\otimes{\rm id}:H^1_{\rm fin}(H_{c,\ell},E[p])[I_{\chi_p}]^\vee\otimes_\chi\W\longrightarrow\Sel_p(E/H_c)[I_{\chi_p}]^\vee\otimes_\chi\W \]
is injective and non-zero.
\end{lemma}

\begin{proof} Keeping in mind the two consequences of Lemma \ref{alg-2} recalled above, proceed as in the proof of \cite[Lemma 8.8]{LV}. \end{proof}

With this auxiliary result at hand, we can prove

\begin{proposition} \label{lemma-II}
If there exists $s\in\Sel_p(E/H_c)[I_{\chi_p}]$ such that $\res_\ell(s)\neq0$ then the map
\[ \res_\ell^\vee\otimes{\rm id}:H^1_{\rm fin}(H_{c,\ell},E[p])^\vee\otimes_\chi\W\longrightarrow\Sel_p(E/H_c)^\vee\otimes_\chi\W \]
is non-zero.
\end{proposition}
\begin{proof} In the commutative square
\[ \xymatrix@C=30pt{H^1_{\rm fin}(H_{c,\ell},E[p])^\vee\otimes_\chi\W\ar[r]^-{\res_\ell^\vee\otimes\text{id}}\ar@{->>}[d]&\Sel_p(E/H_c)^\vee\otimes_\chi\W\ar@{->>}[d]\\
                    H^1_{\rm fin}(H_{c,\ell},E[p])[I_{\chi_p}]^\vee\otimes_\chi\W\ar@{^{(}->}[r]&\Sel_p(E/H_c)[I_{\chi_p}]^\vee\otimes_\chi\W} \]
the vertical maps are surjective and the bottom horizontal arrow is (injective and) non-zero by Lemma \ref{aux-lemma-I}. Hence the upper horizontal arrow must be
non-zero. \end{proof}

\subsection{Construction of an Euler system}

As before, let $\cl O_c$ be the order of $K$ of conductor $c$ and let $H_c$ be the narrow ring class field of $K$ of conductor $c$. Let $\ell$ be an admissible prime such that $p|\ell+1+a_\ell$ (so $\delta=-1$ in Theorem \ref{cong-thm}) and choose $\psi\in\Emb(\cO_c,R)$. Now recall the prime $\lambda_0$ of $H_c$ above $\ell$ fixed in \S \ref{admissible-subsec}; there is a canonical isomorphism
\[ i_{\lambda_0}:H_{c,\lambda_0}\overset\simeq\longrightarrow K_\ell, \]
with $H_{c,\lambda_0}$ being the completion of $H_c$ at $\lambda_0$. Since we are assuming Conjecture \ref{SH-conjecture}, we can consider the Darmon point
\[ P_c=P_\psi^\epsilon\in J^{(\ell)}_\epsilon(H_c)\;\longmono\;J^{(\ell)}_\epsilon(K_\ell), \]
where the injection is induced by $i_{\lambda_0}$. With $\bar\kappa$ as in \eqref{bar-kappa-eq} for $F=H_c$, define a cohomology class
\[ \kappa(\ell):=\bar\kappa(P_c)\in H^1(H_c,E[p]). \]
The collection of classes $\{\kappa(\ell)\}$ indexed by the set of admissible primes is an \emph{Euler system} relative to $E_{/K}$ and, as in \cite{LV}, will be used in the sequel to bound the $p$-Selmer groups. In the following we will deduce the main properties of $\kappa(\ell)$.

Recall the choice of the prime ideal $\fr p$ of $\Z[\chi]$ above $p$ made in \eqref{alg-nonzero-eq}; the ring $\W$ is the completion of $\Z[\chi]$ at $\fr p$. Let us introduce the map
\begin{equation} \label{d-chi-ell-eq}
d_\ell^\chi:H^1(H_c,E[p])\longrightarrow H^1_{\rm sing}(H_{c,\ell},E[p])\otimes_\chi\W
\end{equation}
obtained by composing the restriction from $H^1(H_c,E[p])$ to $H^1(H_{c,\ell},E[p])$ with the map $H^1(H_{c,\ell},E[p])\rightarrow H^1(H_{c,\ell},E[p])\otimes_\chi\W$ which takes $x$ to $x\otimes 1$ and finally with the canonical projection to the singular part of the cohomology.

As explained in \cite[\S 9.3]{LV} (to which we refer for details), the choice of a prime $\lambda_0$ of $H_c$ above $\ell$ made in \S \ref{admissible-subsec} induces natural isomorphisms
\[ H^1_\star(H_{c,\ell},E[p])\overset\simeq\longrightarrow H^1_\star(K_\ell,E[p])\otimes_\Z\Z[G_c] \]
for $\star\in\{\text{fin, sing}\}$, so that we can (and do) view $d_\ell^\chi$ as taking values in the $\W$-module $H^1_{\rm sing}(K_\ell,E[p])\otimes_\Z\W$.

\begin{proposition} \label{nonvanishing-prop}
If $L_K(E,\chi,1)\not=0$ then $d_\ell^\chi\bigl(\kappa(\ell)\bigr)\not=0$.
\end{proposition}

\begin{proof} Let $\iota:\Z[\chi]_S\hookrightarrow\W$ be the natural inclusion (cf. Lemma \ref{completions-lemma}). There is a commutative square
\[ \xymatrix@C=30pt@R=30pt{J_\epsilon^{(\ell)}(K_\ell)\ar[r]^-{\bar\kappa}\ar[d]^-{\partial_\ell}&H^1(K_\ell,E[p])\ar@{->>}[d]^-{\delta_\ell}\\
                           \Phi_{\ell,\epsilon}/\m_{f_\ell}\ar[r]^-{\vartheta_\ell}_-{\simeq}&H^1_{\rm sing}(K_\ell,E[p])} \]
in which $\delta_\ell$ is the projection and $\vartheta_\ell$ is isomorphism \eqref{Phi-sing-isom-eq}. Tensoring with $\Z[\chi]_S$ over $\Z$ and then composing with the relevant maps ${\rm id}\otimes\iota$ yields a commutative diagram
\begin{equation} \label{comm-square-III}
\xymatrix@C=35pt@R=35pt{J_\epsilon^{(\ell)}(K_\ell)\otimes\Z[\chi]_S\ar[r]^-{\bar\kappa\otimes{\rm id}}\ar[d]^-{\partial_\ell\otimes{\rm id}} &
H^1(K_\ell,E[p])\otimes\Z[\chi]_S\ar@{->>}[d]^-{\delta_\ell\otimes{\rm id}}\ar[r]^-{{\rm id}\otimes\iota} & H^1(K_\ell,E[p])\otimes\W \ar@{->>}[d]^-{\delta_\ell\otimes{\rm id}}\\
\Phi_{\ell,\epsilon}/\m_{f_\ell}\otimes\Z[\chi]_S\ar[r]^-{\vartheta_\ell\otimes{\rm id}}_-\simeq\ar[d]^-{{\rm id}\otimes\iota} & {H^1_{\rm sing}(K_\ell,E[p])\otimes\Z[\chi]_S}\ar[r]^-{{\rm id}\otimes\iota} & H^1_{\rm sing}(K_\ell,E[p])\otimes\W \\
\Phi_{\ell,\epsilon}/\m_{f_\ell}\otimes\W\ar[urr]_-{\vartheta_\ell\otimes{\rm id}}^-\simeq}
\end{equation}
The arguments described in \cite[\S\S 9.1--9.3]{LV} show that
\[ d_\ell^\chi\bigl(\kappa(\ell)\bigr)=\bigl((\vartheta_\ell\circ\partial_\ell)\otimes\iota\bigr)(P_\chi^\epsilon) \]
where $P_\chi^\epsilon$ is defined in \eqref{tau-chi}. Since $\vartheta_\ell\otimes{\rm id}$ is an isomorphism, showing that $d_\ell^\chi\bigl(\kappa(\ell)\bigr)\not=0$ is equivalent to showing that
\begin{equation} \label{claim}
\text{$(\partial_\ell\otimes\iota)(P_\chi^\epsilon)\not=0$}\qquad\text{in $\Phi_{\ell,\epsilon}/\m_{f_\ell}\otimes\W\simeq\W/p\W$}
\end{equation}
(here the map $\partial_\ell\otimes\iota$ is equal to the composition of the left vertical arrows in \eqref{comm-square-III}).

In order to prove \eqref{claim} consider the map
\[ \Z[\chi]_S/p\Z[\chi]_S\overset\iota\longrightarrow\W/p\W \]
induced by $\iota$. The non-vanishing of $L_K(E,\chi,1)$ is equivalent, by Theorem \ref{popa-thm}, to the non-vanishing of $\cl L_K(E,\chi,1)$. On the other hand, $\iota\bigl(\bigl[\cl L_K(E,\chi,1)\bigr]\bigr)\not=0$ by \eqref{alg-nonzero-eq} and $p\nmid t_\ell$ because $\ell$ is admissible, hence claim \eqref{claim} follows from Theorem \ref{rec-law}. \end{proof}

\subsection{Local Tate pairings and global duality} \label{tate-subsec}

For every place $v$ of $\Q$, including the archimedean one, denote by
\[ \langle\,,\rangle_v:H^1(H_{c,v},E[p])\times H^1(H_{c,v},E[p])\longrightarrow\Z/p\Z \]
the local Tate pairing at $v$. Global Tate duality, which is a consequence of the reciprocity law of class field theory (specifically, of the global reciprocity law for elements in the Brauer group of $H_c$), asserts that
\begin{equation} \label{GRL1}
\sum_v\langle \res_v(k),\res_v(s)\rangle_v=0
\end{equation}
for all $k,s\in H^1(H_c,E[p])$. Actually, since the Brauer group of $\R$ has order $2$ and $p$ is odd by condition $2$ in Assumption \ref{ass}, for all $k,s\in H^1(H_c,E[p])$ one has
\begin{equation} \label{GRL}
\sum_q\langle\res_q(k),\res_q(s)\rangle_q=0
\end{equation}
with $q$ running over the set of prime numbers (in other words, in \eqref{GRL1} we can restrict the sum to the \emph{finite} places of $\Q$).

Let now $\ell$ be an admissible prime. As explained in \cite[\S 9.4]{LV}, the local Tate pairing $\langle\,,\rangle_\ell$ gives rise to isomorphisms of one-dimensional $\W/p\W$-vector spaces
\begin{equation} \label{iso-tate}
H^1_\star(H_{c,\ell},E[p])\otimes_\chi\W\overset{\simeq}{\longrightarrow}H^1_\bullet(H_{c,\ell},E[p])^\vee\otimes_\chi\W
\end{equation}
for $\{\star,\bullet\}=\{\text{fin, sing}\}$. Moreover, the restriction
\[ \res_\ell:\Sel_p(E/H_c)\longrightarrow H^1_{\rm fin}(H_{c,\ell},E[p]) \]
induces a $\W$-linear map
\[ \eta_\ell:H^1_{\rm sing}(H_{c,\ell},E[p])\otimes_\chi\W\longrightarrow\Sel_p(E/H_c)^\vee\otimes_\chi\W. \]
\begin{lemma} \label{aux-lemma-IV}
If there exists $s\in\Sel_p(E/H_c)[I_{\chi_p}]$ such that $\res_\ell(s)\neq 0$ then $\eta_\ell$ is non-zero.
\end{lemma}
\begin{proof} Immediate from \eqref{iso-tate} and Proposition \ref{lemma-II}. \end{proof}

In the next lemma the symbol $\delta_q$ stands for the local Kummer map at $q$.

\begin{lemma} \label{bad-kummer-lemma}
If $q$ is a prime dividing $N$ then ${\rm Im}(\delta_q)=0$.
\end{lemma}

\begin{proof} Since $\delta_q$ factors through $E(H_{c,q})/pE(H_{c,q})$, the statement follows from condition $5$ in Assumption \ref{ass}. \end{proof}

Now recall the map $d_\ell^\chi$ defined in \eqref{d-chi-ell-eq}.

\begin{proposition} \label{kernel-eta-prop}
The element $d_\ell^\chi\bigl(\kappa(\ell)\bigr)$ belongs to the kernel of $\eta_\ell$.
\end{proposition}

\begin{proof} Keeping Lemma \ref{bad-kummer-lemma} and formula \eqref{GRL} in mind, proceed exactly as in the proof of \cite[Proposition 9.6]{LV}. \end{proof}

\subsection{Proof of the first vanishing result} \label{main-subsec}

As a first arithmetic consequence of Theorem \ref{rec-law}, we prove a vanishing result for twisted Selmer groups: all other results will follow from this one. Recall that we are assuming Conjecture \ref{SH-conjecture} throughout.

\begin{theorem} \label{main-thm}
If $L_K(E,\chi,1)\not=0$ then $\Sel_p(E/H_c)\otimes_\chi\W=0$.
\end{theorem}

\begin{proof} By what was said in \S \ref{preliminaries-subsec}, it is enough to show that $\Sel_p(E/H_c)[I_{\chi_p}]=0$. Assume that $s\in\Sel_p(E/H_c)[I_{\chi_p}]$ is not zero and choose an admissible prime $\ell$ such that $p|a_\ell+\ell+1$ and $\res_\ell(s)\not=0$, which exists by Proposition \ref{existence-admissible-primes}. Since $L_K(E,\chi,1)\not=0$, Proposition \ref{nonvanishing-prop} ensures that $d_\ell^\chi\bigl(\kappa(\ell)\bigr)\not=0$; then $d_\ell^\chi\bigl(\kappa(\ell)\bigr)$ generates $H^1_{\rm sing}(H_{c,\ell},E[p])\otimes_\chi\W$ over $\W$. On the other hand, Proposition \ref{kernel-eta-prop} says that $d_\ell^\chi\bigl(\kappa(\ell)\bigr)$ belongs to the kernel of the $\W$-linear map $\eta_\ell$, and this contradicts the non-triviality of $\eta_\ell$ that was shown in Lemma \ref{aux-lemma-IV}. \end{proof}

By exploiting the surjectivity of the representation $\rho_{E,p}$ (condition $3$ in Assumption \ref{ass}) and the flatness of $\W$ over $\Z[G_c]$ (Lemma \ref{alg-2}), formal algebraic considerations yield also the following reformulation of Theorem \ref{main-thm}.

\begin{theorem} \label{main2-thm}
If $L_K(E,\chi,1)\not=0$ then
\[ \Sel_{p^n}(E/H_c)\otimes_\chi\W=0 \]
for all integers $n\geq1$.
\end{theorem}

The reader is referred to \cite[Theorem 9.8]{LV} for details.

\subsection{Applications} \label{applications-subsec}

In this subsection let $K'$ be an extension of $K$ contained in $H_c$ and let
\[ \lambda:\Gal(K'/K)\longrightarrow\C^\times \]
be a character. Adopting the usual notation for twisted $L$-functions and eigenspaces, the first consequence of Theorem \ref{main-thm} is the following

\begin{theorem} \label{cons1-thm}
If $L_K(E,\lambda,1)\not=0$ then $E(K')^\lambda=0$.
\end{theorem}

\begin{proof} Let $\chi\in\widehat G_c$ be the character induced by $\lambda$ in the obvious way, so that there is an equality of twisted $L$-functions
\[ L_K(E,\chi,s)=L_K(E,\lambda,s) \]
up to finitely many Euler factors (cf., e.g., \cite[\S 7]{zh2}). Therefore $L_K(E,\chi,1)\not=0$, whence $E(H_c)^\chi=0$ by a combination of Proposition \ref{alg-4} and Theorem \ref{main-thm}. But there is a natural inclusion $E(K')^\lambda\subset E(H_c)^\chi$, and the theorem is proved. \end{proof}

Theorem \ref{cons1-thm} is the $\lambda$-twisted conjecture of Birch and Swinnerton-Dyer for $E$ over $K'$ in the case of analytic rank $0$. In fact, under this analytic condition Theorem \ref{main-thm} also yields a vanishing result for the groups $\Sel_p(E/K')$ for all prime numbers $p$ satisfying Assumption \ref{ass} (recall that this excludes only finitely many primes). As will be clear, to obtain this it is crucial that we were able to prove Theorem \ref{main-thm} for \emph{all} complex-valued characters $\chi$ of $G_c$.

To begin with, we need some further notation and an auxiliary result. Let $\Q_p^{\rm nr}$ be the maximal unramified extension of $\Q_p$, let $\cO_{\Q_p^{\rm nr}}$ be its ring of integers and let $\kappa_p$ be its residue field (which is an algebraic closure of $\F_p$). In order to avoid confusion, for every $\chi\in\widehat G_c$ denote $\W_\chi$ the ring $\W$ associated with $\chi$ as in \S \ref{preliminaries-subsec}. Finally, since every $\W_\chi$ is a finite unramified extension of $\Z_p$, for all $\chi$ we can (and do) fix embeddings $\W_\chi\hookrightarrow\cO_{\Q_p^{\rm nr}}$, which endow $\kappa_p$ with a structure of $\W_\chi$-module. Then define
\[ \Sel_p(E/K')^\lambda:=\bigl\{x\in \Sel_p(E/K')\otimes_\Z\kappa_p\mid\text{$\sigma(x)=\lambda(\sigma)x$ for all $\sigma\in\Gal(K'/K)$}\bigr\}. \]
From here on let $p$ be a prime satisfying Assumption \ref{ass}.

\begin{lemma} \label{eigen-sel-lem}
If $L_K(E,\lambda,1)\not=0$ then $\Sel_p(E/K')^\lambda=0$.
\end{lemma}

\begin{proof} Let $\chi\in\widehat G_c$ be the character induced by $\lambda$. Then, as in the proof of Theorem \ref{cons1-thm}, $L_K(E,\chi,1)\not=0$, whence $\Sel_p(E/H_c)\otimes_\chi\W_\chi=0$ by Theorem \ref{main-thm}. Since $p\nmid h^+(c)$, one can apply Maschke's theorem to the $G_c$-representation $\Sel_p(E/H_c)\otimes_\Z\kappa_p$ and mimic the proof of \cite[Proposition 8.1]{LV} to obtain an identification
\[ \Sel_p(E/H_c)^\chi=\Sel_p(E/H_c)\otimes_\chi\kappa_p \]
of $\kappa_p[G_c]$-modules. Thus we get that
\begin{equation} \label{eq1-eq}
\Sel_p(E/H_c)^\chi=\bigl(\Sel_p(E/H_c)\otimes_\chi\W_\chi\bigr)\otimes_{\W_\chi}\kappa_p=0.
\end{equation}
On the other hand, as explained in \cite[Lemma 4.3]{Gross}, the surjectivity of $\rho_{E,p}$ ensures that $E$ has no non-trivial $p$-torsion rational over $H_c$, and then the inflation-restriction exact sequence in Galois cohomology gives an injection $\Sel_p(E/K')\hookrightarrow\Sel_p(E/H_c)$, which in turn induces an injection
\begin{equation} \label{eq2-eq}
\Sel_p(E/K')^\lambda\;\longmono\;\Sel_p(E/H_c)^\chi
\end{equation}
of eigenspaces. The lemma follows by combining \eqref{eq1-eq} and \eqref{eq2-eq}. \end{proof}

Let now $L_{K'}(E,s)$ be the $L$-function of $E$ over $K'$.

\begin{theorem} \label{cons2-thm}
If $L_{K'}(E,1)\not=0$ then
\[ \Sel_{p^n}(E/K')=0 \]
for all integers $n\geq1$.
\end{theorem}

\begin{proof} Routine algebraic considerations show that it is enough to prove the result for $n=1$. For simplicity, set $G':=\Gal(K'/K)$. There is a factorization
\begin{equation} \label{L-fact-eq}
L_{K'}(E,s)=\prod_\lambda L_K(E,\lambda,s)
\end{equation}
where $\lambda$ varies over the complex-valued characters of $G'$. Now observe that the embeddings $\W_\chi\hookrightarrow\cO_{\Q_p^{\rm nr}}$ fixed before induce a bijection between the $\kappa_p$-valued and the $\C$-valued characters of $G'$. Therefore, since $p\nmid[K':K]$, Maschke's theorem ensures that there is a decomposition
\begin{equation} \label{eigen-decomp-eq}
\Sel_p(E/K')\otimes_\Z\kappa_p=\bigoplus_\lambda\Sel_p(E/K')^\lambda
\end{equation}
as a direct sum of eigenspaces. Since $L_{K'}(E,1)\not=0$, equality \eqref{L-fact-eq} implies that $L_K(E,\lambda,1)\not=0$ for all $\lambda$, hence $\Sel_p(E/K')^\lambda=0$ for all $\lambda$ by Lemma \ref{eigen-sel-lem}. Since $\Sel_p(E/K')$ is a finite-dimensional $\F_p$-vector space, the theorem is an immediate consequence of \eqref{eigen-decomp-eq}. \end{proof}

As a piece of notation, for every integer $n\geq1$ let $\Sha_{p^n}(E/K')$ be the $p^n$-Shafarevich--Tate group of $E$ over $K'$. Theorem \ref{cons2-thm} immediately yields

\begin{corollary} \label{cons2-coro}
If $L_{K'}(E,1)\not=0$ then $\Sha_{p^n}(E/K')=0$ for all $n\geq1$ and $E(K')$ is finite.
\end{corollary}

This is the conjecture of Birch and Swinnerton-Dyer for $E$ over $K'$ in analytic rank $0$.

\begin{remarkwr}
1) The Birch and Swinnerton-Dyer conjecture for $E$ over $K'$ in analytic rank $0$ can also be obtained directly from Theorem \ref{cons1-thm} via a decomposition argument analogous to the one used in the proof of Theorem \ref{cons2-thm}.

2) If $K'=K$ then Theorem \ref{cons2-thm} is part of a result due to Kolyvagin (a sketch of proof of which can be found in \cite[Theorem 9.11]{LV}) establishing (unconditionally) the finiteness of $E(K)$ and $\Sha(E/K)$ for all quadratic fields $K$ such that $L_K(E,1)\not=0$. The key ingredients in Kolyvagin's proof of this theorem are non-vanishing results for the special values of the first derivatives of base changes of $L(E,s)$ to suitable auxiliary \emph{imaginary} quadratic fields and Kolyvagin's results in rank one. In light of this, even in the particular case where $K'=K$ our proof of Theorem \ref{cons2-thm}, albeit conditional, is genuinely new, since it takes place entirely ``in rank zero'' and in the \emph{real} quadratic setting, without invoking any result over imaginary quadratic fields.

3) It should be possible, with some extra effort, to extend the techniques of this article and obtain the finiteness of the full Shafarevich--Tate groups $\Sha(E/K')$.
\end{remarkwr}

\end{document}